\documentclass[a4paper]{amsart}

\usepackage[utf8]{inputenc}
\usepackage[T1]{fontenc}
\usepackage{lmodern}
\usepackage{amsmath}
\usepackage{amsfonts}
\usepackage{amssymb}
\usepackage{amsthm}
\usepackage{thmtools}
\usepackage{mathtools}
\usepackage{tikz-cd}
\usepackage{multirow}
\usepackage[pdftex,
            pdfauthor={Jethro van Ekeren, Sven Möller and Nils R. Scheithauer},
            pdftitle={Dimension Formulae in Genus Zero and Uniqueness of Vertex Operator Algebras},
            pdfsubject={},
            pdfkeywords={},
            pdfproducer={},
            pdfcreator={}]{hyperref}

\theoremstyle{plain}
\newtheorem{thm}{Theorem}[section]
\newtheorem{prop}[thm]{Proposition}
\newtheorem{cor}[thm]{Corollary}
\newtheorem{lem}[thm]{Lemma}
\newtheorem{conj}[thm]{Conjecture}
\newtheorem*{thm*}{Theorem}

\theoremstyle{definition}

\newtheorem*{nota*}{Notation}
\newtheorem{rem}[thm]{Remark}

\newcommand{\Q}{\mathbb{Q}}
\newcommand{\Z}{\mathbb{Z}}
\newcommand{\Ns}{\mathbb{Z}_{>0}}
\newcommand{\N}{\mathbb{Z}_{\geq0}}
\newcommand{\C}{\mathbb{C}}
\newcommand{\R}{\mathbb{R}}
\renewcommand{\H}{\mathbb{H}}

\newcommand{\tr}{\operatorname{tr}}
\renewcommand{\i}{\mathrm{i}}
\newcommand{\e}{\mathrm{e}}

\newcommand{\Aut}{\operatorname{Aut}}

\newcommand{\ord}{\operatorname{ord}}

\newcommand{\voa}{vertex operator algebra}
\newcommand{\Voa}{Vertex operator algebra}
\newcommand{\VOA}{Vertex Operator Algebra}
\newcommand{\vosa}{vertex operator subalgebra}

\newcommand{\fpvosa}{fixed-point vertex operator subalgebra}

\newcommand{\vac}{\textbf{1}}
\newcommand{\ch}{\operatorname{ch}}

\newcommand{\id}{\operatorname{id}}
\newcommand{\amgis}{\zeta}
\newcommand{\eps}{\varepsilon}
\newcommand{\lcm}{\operatorname{lcm}}
\newcommand{\SLZ}{\operatorname{SL}_2(\mathbb{Z})}
\newcommand{\GL}{\operatorname{GL}}

\newcommand{\ee}{\mathfrak{e}}
\newcommand{\g}{\mathfrak{g}}
\newcommand{\hh}{\mathfrak{h}}
\newcommand{\h}{{L_\C}}
\newcommand{\ad}{\operatorname{ad}}
\newcommand{\Inn}{\operatorname{Inn}}
\newcommand{\Out}{\operatorname{Out}}

\newcommand{\orb}{\operatorname{orb}}

\newcommand{\strathol}{strongly rational, holomorphic}
\newcommand{\strat}{strongly rational}
\newcommand{\II}{I\!I}

\newlength{\myl}
\newcommand{\s}{\hspace{\myl}}

\newcommand{\maintheorem}{%
A \strathol{} \voa{} of central charge 24 is up to isomorphism uniquely determined by the Lie algebra structure of its weight-one space $V_1$ if $V_1$ is isomorphic to $A_5C_5E_{6,2}$, $A_3A_{7,2}C_3^2$, $A_{8,2}F_{4,2}$, $B_8E_{8,2}$, $A_2^2A_{5,2}^2B_2$, $C_8F_4^2$, $A_{4,2}^2C_{4,2}$, $A_{2,2}^4D_{4,4}$, $B_5E_{7,2}F_4$, $B_4C_6^2$, $A_{4,5}^2$, $A_4A_{9,2}B_3$, $B_6C_{10}$, $A_1C_{5,3}G_{2,2}$ or $A_{1,2}A_{3,4}^3$, corresponding to the cases $44$, $33$, $36$, $62$, $26$, $52$, $22$, $13$, $53$, $48$, $9$, $40$, $56$, $21$ and $7$ on Schellekens' list.
}

\begin{document}

\title[Dimension Formulae in Genus Zero and Uniqueness of VOAs]{Dimension Formulae in Genus Zero and Uniqueness of Vertex Operator Algebras}
\author[Jethro van Ekeren, Sven Möller and Nils~R.\ Scheithauer]{Jethro van Ekeren\textsuperscript{\lowercase{a}}, Sven Möller\textsuperscript{\lowercase{b},\lowercase{c}} and Nils~R.\ Scheithauer\textsuperscript{\lowercase{b}}}
\thanks{\textsuperscript{a}{Universidade Federal Fluminense, Niterói, RJ, Brazil}}
\thanks{\textsuperscript{b}{Technische Universität Darmstadt, Darmstadt, Germany}}
\thanks{\textsuperscript{c}{Rutgers University, Piscataway, NJ, United States of America}}
\thanks{Email: \href{mailto:jethrovanekeren@gmail.com}{\nolinkurl{jethrovanekeren@gmail.com}}, \href{mailto:math@moeller-sven.de}{\nolinkurl{math@moeller-sven.de}}}

\begin{abstract}
We prove a dimension formula for orbifold \voa{}s of central charge 24 by automorphisms of order~$n$ such that $\Gamma_0(n)$ is a genus zero group. We then use this formula together with the inverse orbifold construction for automorphisms of orders $2$, $4$, $5$, $6$ and $8$ to establish that each of the following fifteen Lie algebras is the weight-one space $V_1$ of exactly one holomorphic, $C_2$-cofinite \voa{} $V$ of CFT-type and central charge 24: $A_5C_5E_{6,2}$, $A_3A_{7,2}C_3^2$, $A_{8,2}F_{4,2}$, $B_8E_{8,2}$, $A_2^2A_{5,2}^2B_2$, $C_8F_4^2$, $A_{4,2}^2C_{4,2}$, $A_{2,2}^4D_{4,4}$, $B_5E_{7,2}F_4$, $B_4C_6^2$, $A_{4,5}^2$, $A_4A_{9,2}B_3$, $B_6C_{10}$, $A_1C_{5,3}G_{2,2}$ and $A_{1,2}A_{3,4}^3$.
\end{abstract}

\maketitle

\setcounter{tocdepth}{1}
\tableofcontents

\section{Introduction}

\Voa{}s and their representation theory formalise the physical notions of chiral algebra and fusion ring of a conformal field theory (CFT). If the fusion ring is as simple as possible, then the conformal field theory is said to be holomorphic. More precisely, a rational \voa{} is said to be \emph{holomorphic} if it possesses a unique irreducible module, namely itself.

In this paper we follow the convention in \cite{DM04b} and call a \voa{} \emph{\strat{}} if it is rational (as defined in \cite{DLM97}, for example), $C_2$-cofinite, self-contragredient (or self-dual) and of CFT-type.

The central charge $c$ of a \strathol{} \voa{} $V$ is necessarily a positive integer multiple of $8$ (a simple consequence of \cite{Zhu96}), and a complete classification is known in the cases $c=8,16$ \cite{DM04}. The classification at central charge $c=24$ is an open problem, but a lot of progress has been made recently.

Let $V$ be a \strathol{} \voa{} and $\sigma$ a finite order automorphism of $V$. In \cite{EMS15} we study the fusion ring of the \fpvosa{} $V^\sigma$, showing that it is naturally isomorphic to the group ring of a finite quadratic space and that Zhu's representation of $\SLZ$ on the trace functions of the irreducible $V^\sigma$-modules \cite{Zhu96} coincides up to a character with the corresponding Weil representation. In the same paper we develop a cyclic orbifold theory and use it to construct a new holomorphic \voa{}, denoted by $V^{\orb(\sigma)}$, for appropriately chosen $\sigma$.

To fully exploit the power of the orbifold theory as a tool for the construction and classification of holomorphic \voa{}s it is important to be able to deduce information about the isomorphism type of $V^{\orb(\sigma)}$ from $V$ and $\sigma$. One useful invariant of a \voa{} of central charge $24$ is the dimension of its weight-one Lie algebra $V_1$. In Theorem~\ref{thm:dimform} we prove a formula for the dimension of this Lie algebra $V_1^{\orb(\sigma)}$ in terms of $V$ and $\sigma$. The following is a useful corollary. For more details see Section~\ref{sec:dimform}.
\begin{thm*}[Dimension Formula]
Let $n$ be a positive integer such that the congruence subgroup $\Gamma_0(n)\subseteq\SLZ$ has genus zero. Let $V$ be a \strathol{} \voa{} of central charge $24$ and $\sigma$ an automorphism of $V$ of type $n\{0\}$ such that the conformal weights of the irreducible $\sigma^i$-twisted $V$-modules obey $\rho(V(\sigma^i))\geq 1$ for $i=1,\ldots,n-1$. Then
\begin{equation*}
\sum_{d\mid n}\frac{\phi((d,n/d))}{(d,n/d)}\left(24+\frac{n}{d}\dim(V_1^\sigma)-\dim(V_1^{\orb(\sigma^d)})\right)=24
\end{equation*}
holds where $\phi$ is Euler's totient function.
\end{thm*}

In \cite{EMS15} we also show that the weight-one space $V_1$ of a \strathol{} \voa{} $V$ of central charge $c=24$ must be one of an explicit list of $71$ Lie algebras, rigorously proving a result by Schellekens on ``meromorphic $c=24$ conformal field theories'' \cite{Sch93}. We shall in the following refer to these 71 cases as \emph{Schellekens' list}.

Efforts by many authors have finally led to the result that all of these 71 Lie algebras on Schellekens' list are realised as weight-one spaces of \strathol{} \voa{}s of central charge 24 \cite{FLM88,DGM90,Don93,DGM96,Lam11,LS12a,LS15a,Miy13b,SS16,EMS15,Moe16,LS16a,LS16,LL16}. All are obtained as lattice \voa{}s and their (iterated) orbifolds by finite order automorphisms. Some of these results are based on the general theory of cyclic orbifolds developed in \cite{EMS15,Moe16}.

In order to obtain a classification of \voa{}s it is necessary to show that the Lie algebra structure of the weight-one space of a \strathol{} \voa{} of central charge 24 uniquely determines this \voa{} up to isomorphism. This is known for the 24 Niemeier lattice \voa{}s \cite{DM04}, 3 \voa{}s obtained as orbifolds of order~3 \cite{LS16b}, 13 \voa{}s obtained as $(-1)$-orbifolds \cite{KLL18}, 2 framed \voa{}s \cite{LS15a} and for one more case \cite{LL16}.

In this paper we add 14 more cases to this list (see Theorem~\ref{thm:main} and below).  The uniqueness of 5 further cases was proved in \cite{LS17} after the completion of this work. In total, this leaves 9 open cases, including the Moonshine module $V^\natural$. It is a decades-old conjecture that $V^\natural$ is the only (\strat{}) holomorphic \voa{} $V$ of central charge 24 with $V_1=\{0\}$ \cite{FLM88}. The methods used in this paper are, however, unsuited to prove this conjecture.
\begin{thm*}[Uniqueness]
\maintheorem
\end{thm*}

The strategy we use to prove uniqueness is developed in \cite{LS16b} and is based in an essential way on the \emph{inverse orbifold construction}, first developed in \cite{EMS15} (and called ``reverse orbifold'' in \cite{LS16b}). In a nutshell, the idea is as follows: Let $V$ be a \strathol{} \voa{} of central charge $24$ with weight-one Lie algebra $V_1$. Any inner automorphism $\sigma=\e^{(2\pi\i)\ad_v}$ of $V_1$, $v\in V_1$, extends to an automorphism of $V$. Choose such an automorphism, and suppose that $V^{\orb(\sigma)}$ is isomorphic to the lattice \voa{} $V_L$. By virtue of its construction as an orbifold, $V^{\orb(\sigma)}$ carries an inverse orbifold automorphism $\amgis$ of the same order as $\sigma$ with $V=(V^{\orb(\sigma)})^{\orb(\amgis)}$. Under favourable circumstances the order and the fixed-point subalgebra $(V^{\orb(\sigma)})_1^{\amgis}=V_1^\sigma$ of $\amgis$ determine $\amgis$ up to conjugacy in $\Aut(V_L)$. If this is the case, then $V\cong(V_L)^{\orb(\amgis)}$ is shown to be unique up to isomorphism.

\subsection*{Outline}
We assume that the reader is familiar with \voa{}s and their representation theory and with the specific examples provided by lattice \voa{}s and affine \voa{}s (see, e.g.\ \cite{LL04}).

The article is organised as follows: In Section~\ref{sec:voalie} we describe Lie algebras occurring as weight-one spaces of \voa{}s, inner automorphisms and the classification of the weight-one structures of \strathol{} \voa{}s of central charge 24, dubbed Schellekens' list.

Section~\ref{sec:orbifold} reviews the cyclic orbifold construction of \strathol{} \voa{}s and the inverse orbifold construction. In Section~\ref{sec:dimform} a genus zero dimension formula for the weight-one space of these orbifolds for central charge 24 is derived.

In Section~\ref{sec:uniqueness} the general procedure to prove the uniqueness of certain \voa{}s on Schellekens' list is described. This is split up into three parts, which are described in detail in Sections \ref{sec:part1}, \ref{sec:part2} and \ref{sec:part3} for the fifteen cases considered in this text.

Finally, in Section~\ref{sec:main} we state the main result of this paper.

\subsection*{Acknowledgements}
The authors would like to thank Ching Hung Lam and Hiroki Shimakura for helpful discussions. The second and third author both were partially supported by the DFG project ``Infinite-dimensional Lie algebras in string theory''. A part of the work was done while the first and second author visited the ESI in Vienna for the conference ``Geometry and Representation Theory'' in 2017. They are grateful to both the ESI and to the organisers of the conference. We thank the two anonymous referees for their comments and suggestions.

\subsection*{Computer Calculations}
Sections \ref{sec:part1} and \ref{sec:part3} rely on information about the automorphism groups of the Niemeier lattices, certain subgroups and quotient groups and their decomposition into conjugacy classes, which is partially obtained using the computer algebra system \texttt{Magma} \cite{Magma}. Moreover, the computation of the lattice orbifolds in Section~\ref{sec:part1} makes use of \texttt{Sage} \cite{sage} as described in Section~5.6 of \cite{Moe16}.

\section{\VOA{}s and Lie Algebras}\label{sec:voalie}
We briefly review the appearance of Lie algebras as weight-one spaces of \voa{}s and discuss related concepts such as affine structures and inner automorphisms.

\subsection{Lie Algebras}\label{sec:la}

Let $\g$ be a finite-dimensional, simple Lie algebra over $\C$ of rank $l\in\Ns$. Let $\hh\subseteq\g$ be a Cartan subalgebra, $\Phi\subseteq\hh^*$ the corresponding root system and $\Delta=\{\alpha_1,\ldots,\alpha_l\}\subseteq\Phi$ a base of simple roots, labelled as in \cite{Hum73}, p.~58.
We denote by $(\cdot,\cdot)$ the unique non-degenerate, symmetric, invariant bilinear form on $\g$ normalised so that the induced form on $\hh^*$, which we also denote by $(\cdot,\cdot)$, satisfies $(\alpha,\alpha)=2$ for any long root $\alpha\in\Phi$.

The coroot $\alpha^\vee\in\hh$ associated with $\alpha\in\Phi$ is defined by $\beta(\alpha^\vee)=2(\alpha,\beta)/(\alpha,\alpha)$ for all $\beta\in\hh^*$. The set $\Phi^\vee\subseteq\hh$ of coroots is also a root system, of which the set $\Delta^\vee=\{\alpha^\vee\,|\,\alpha\in\Delta\}$ is a base. The root lattice $Q$ and the coroot lattice $Q^\vee$ are the $\Z$-spans of $\Phi$ and $\Phi^\vee$, respectively.

The lattice $P\subseteq\hh^*$ dual to $Q^\vee\subseteq\hh$ is called the weight lattice, and the elements of the basis $\Lambda_1,\ldots,\Lambda_l\in\hh^*$ dual to the simple coroots $\alpha_1^\vee,\ldots,\alpha_l^\vee\in\hh$ are called the fundamental weights. Similarly, the fundamental coweights $\Lambda_1^\vee,\ldots,\Lambda_l^\vee\in\hh$ are defined to be dual to the simple roots and span the coweight lattice. By definition $2(\Lambda_i,\alpha_j)/(\alpha_j,\alpha_j)=\Lambda_i(\alpha_j^\vee)=\delta_{i,j}$ holds for $i,j=1,\ldots,l$.

Let $\delta:=\sum_{i=1}^l\Lambda_i$ be the sum of all fundamental weights and $\delta^\vee:=\sum_{i=1}^l\Lambda_i^\vee$ the sum of the fundamental coweights. Among the roots $\alpha=\sum_{i=1}^lm_i\alpha_i$ of $\g$, the unique one with the highest value of $\sum_{i=1}^lm_i$ is denoted $\theta$ and called the highest root. The highest root is always a long root. The relation $\theta=\sum_{i=1}^la_i\alpha_i$ defines the ``marks'' $a_1,\ldots,a_l$ of $\g$. Similarly, the corresponding coroot $\theta^\vee$ defines the ``comarks'' $a_1^\vee,\ldots,a_l^\vee$ of $\g$ by $\theta^\vee=\sum_{i=1}^la_i^\vee\alpha_i^\vee$. The Coxeter number and dual Coxeter number are defined as $h:=1+\sum_{i=1}^la_i$ and $h^\vee:=1+\sum_{i=1}^la_i^\vee$, respectively.

We denote by $P_+\subseteq\hh^*$ the set of dominant integral weights, i.e.\ the non-negative integer linear combinations of the fundamental weights $\Lambda_i$. The level of a weight $\mu\in P_+$ is by definition $\mu(\theta^\vee)$. For $k\in\N$ let $P_+^k:=\{\mu\in P_+\,|\,\mu(\theta^\vee)\leq k\}$ denote the subset of $P_+$ of dominant integral weights of level at most $k$. Note that the level of $\delta$ is $h^\vee-1$.

The reflection $r_\alpha$ in the root $\alpha\in\Phi$ is the isometry of $\hh^*$ given by $r_\alpha(\beta)=\beta-\beta(\alpha^\vee)\alpha$ for all $\beta\in\hh^*$, and the Weyl group $W\subseteq\GL(\hh^*)$ is the group generated by the reflections in all roots or equivalently in all simple roots. $W$ fixes $\Phi$, $Q$ and $P$ as sets. Via the isometry $\hh\to\hh^*$ induced by $(\cdot,\cdot)$ the Weyl group $W$ also acts on $\hh$ and fixes $\Phi^\vee$, $Q^\vee$ and the coweight lattice.

We will sometimes write $\Phi(\g)$, $\Delta(\g)$, $(\cdot,\cdot)_\g$, etc.\ in order to specify the simple Lie algebra $\g$.

For a dominant integral weight $\lambda\in P_+$ let $L_\g(\lambda)$ be the up to isomorphism unique finite-dimensional, irreducible highest-weight $\g$-module with highest weight $\lambda$.
For any finite-dimensional $\g$-module $M$ we denote by $\Pi(M)\subseteq P$ the set of weights of $M$, i.e.\ $M$ has weight decomposition $M=\bigoplus_{\mu\in\Pi(M)}M_\mu$. We also denote $\Pi(L_\g(\lambda))$ by $\Pi(\lambda,\g)$.

\subsection{Weight-One Lie Algebra}
Let $V=\bigoplus_{n=0}^\infty V_n$ be a \voa{} of CFT-type. It is well-known that
\begin{equation*}
[a,b]:=a_0b
\end{equation*}
for $a,b\in V_1$ endows $V_1$ with the structure of a Lie algebra and that the zero modes $a_0$ for $a\in V_1$ equip each $V$-module with a weight-preserving action of this Lie algebra.

If $V$ is also self-contragredient, then there exists a non-degenerate, invariant bilinear form $\langle\cdot,\cdot\rangle$ on $V$, which is unique up to a non-zero scalar and symmetric \cite{FHL93,Li94}. We normalise the form so that $\langle\vac,\vac\rangle=-1$, where $\vac$ is the vacuum vector of $V$. Then, for $a,b\in V_1$,
\begin{equation*}
a_1b=b_1a=\langle a,b\rangle\vac.
\end{equation*}
Note that $V_n$ and $V_m$ are orthogonal with respect to $\langle\cdot,\cdot\rangle$ for $m\neq n$.

If $g\in\Aut(V)$ is an automorphism of the \voa{} $V$ (which by definition fixes the vacuum vector $\vac\in V_0$ and the Virasoro vector $\omega\in V_2$), then the restriction of $g$ to $V_1$ is a Lie algebra automorphism, possibly of smaller order.

\subsection{Affine Structure}\label{sec:affine}
Let $\g$ be a simple, finite-dimensional Lie algebra with invariant bilinear form $(\cdot,\cdot)_\g$, normalised as described above. Then the affine Kac-Moody algebra $\hat\g$ associated with $\g$ is the Lie algebra $\hat\g:=\g\otimes\C[t,t^{-1}]\oplus\C K$ with central element $K$ and Lie bracket
\begin{equation}\label{eq:affine}
[a\otimes t^m,b\otimes t^n]:=[a,b]\otimes t^{m+n}+m(a,b)_\g\delta_{m+n,0}K
\end{equation}
for $a,b\in\g$, $m,n\in\Z$.

A representation of $\hat\g$ is said to have level $k\in\C$ if $K$ acts as $k\id$. For a dominant integral weight $\lambda\in P_+(\g)$ and $k\in\C$ let $L_{\hat\g}(k,\lambda)$ be the irreducible $\hat\g$-module of level $k$ obtained by inducing the irreducible highest-weight $\g$-module $L_\g(\lambda)$ up in a certain way to a $\hat\g$-module and taking its irreducible quotient. For more details on affine Kac-Moody algebras and their modules we refer the reader to \cite{Kac90}.

For $k\in\Ns$, $L_{\hat\g}(k,0)$ admits the structure of a rational \voa{} whose irreducible modules are given by the $L_{\hat\g}(k,\lambda)$ for $\lambda\in P_+^k(\g)$, the subset of the dominant integral weights $P_+(\g)$ of level at most $k$ (see \cite{FZ92}, Theorem~3.1.3 or, e.g.\ \cite{LL04}, Sections 6.2 and 6.6). The conformal weight of $L_{\hat\g}(k,\lambda)$ is
\begin{equation*}
\rho(L_{\hat\g}(k,\lambda))=\frac{(\lambda+2\delta,\lambda)_\g}{2(k+h^\vee)}
\end{equation*}
(see \cite{Kac90}, Corollary~12.8).

For a self-contragredient \voa{} $V$ of CFT-type 
\begin{equation*}
[a_m,b_n]=(a_0b)_{m+n}+m(a_1b)_{m+n-1}=[a,b]_{m+n}+m\langle a,b\rangle\delta_{m+n,0}\id_V
\end{equation*}
for all $a,b\in V_1$, $m,n\in\Z$. Comparing this with \eqref{eq:affine} we see that for a simple Lie subalgebra $\g$ of $V_1$ the map $a\otimes t^n\mapsto a_n$ for $a\in\g$ and $n\in\Z$ defines a representation of $\hat\g$ on $V$ of some level $k_\g\in\C$ with $\langle\cdot,\cdot\rangle=k_\g(\cdot,\cdot)_\g$ restricted to $\g$.

In the following let us assume that $V$ is simple and \strat{}. Then it is shown in \cite{DM04b} that the Lie algebra $V_1$ is reductive, i.e.\ a direct sum of a semisimple and an abelian Lie algebra. Moreover, Theorem~3.1 in \cite{DM06b} states that for a simple Lie subalgebra $\g$ of $V_1$ the restriction of $\langle\cdot,\cdot\rangle$ to $\g$ is non-degenerate, the level $k_\g$ is a positive integer, the
\vosa{} of $V$ generated by $\g$ is isomorphic to $L_{\hat\g}(k_\g,0)$ and $V$ is an integrable $\hat\g$-module.

Now suppose that $c=24$ and $V$ is holomorphic. It is proved in \cite{DM04} that the Lie algebra $V_1$ is zero, abelian of dimension 24 or semisimple of rank at most $24$.
Let us assume that $V_1$ is semisimple. Then $V_1$ decomposes into a direct sum
\begin{equation*}
V_1=\g_1\oplus\ldots\oplus\g_r
\end{equation*}
of simple Lie algebras $\g_i$ and the \vosa{} $\langle V_1\rangle$ of $V$ generated by $V_1$ is isomorphic to
\begin{equation*}
\langle V_1\rangle\cong L_{\hat\g_1}(k_1,0)\otimes\ldots\otimes L_{\hat\g_r}(k_r,0)
\end{equation*}
with levels $k_i:=k_{\g_i}\in\Ns$ and has the same Virasoro vector as $V$. This decomposition of the \voa{} $\langle V_1\rangle$ is called the \emph{affine structure} of $V$. We denote it by the symbol
\begin{equation*}
\g_{1,k_1}\ldots\g_{r,k_r}
\end{equation*}
and sometimes omit $k_i$ if it equals 1. By abuse of language we will also refer to the cases of abelian and zero Lie algebras $V_1$ as certain affine structures.

Since $\langle V_1\rangle\cong L_{\hat\g_1}(k_1,0)\otimes\ldots\otimes L_{\hat\g_r}(k_r,0)$ is rational,
$V$ decomposes into the direct sum of finitely many irreducible $\langle V_1\rangle$-modules
\begin{equation}\label{eq:Vdecomp}
V\cong\bigoplus_{(\lambda_1,\ldots,\lambda_r)}m_{(\lambda_1,\ldots,\lambda_r)}L_{\hat\g_1}(k_1,\lambda_1)\otimes\ldots\otimes L_{\hat\g_r}(k_r,\lambda_r)
\end{equation}
where the $m_{(\lambda_1,\ldots,\lambda_r)}\in\N$ and the sum runs over finitely many tuples $(\lambda_1,\ldots,\lambda_r)$ of dominant integral weights with each $\lambda_i\in P_+^{k_i}(\g_i)$, i.e.\ of level at most $k_i$. Note that the conformal weight of a module $M=L_{\hat\g_1}(k_1,\lambda_1)\otimes\ldots\otimes L_{\hat\g_r}(k_r,\lambda_r)$ is
\begin{equation}\label{eq:confsum}
\rho(M)=\sum_{i=1}^r\frac{(\lambda_i+2\delta_i,\lambda_i)_{\g_i}}{2(k_i+h_i^\vee)}.
\end{equation}
Since $\langle V_1\rangle_1=V_1$, except for $(\lambda_1,\ldots,\lambda_r)=(0,\ldots,0)$ the conformal weights of the tensor-product modules in the above decomposition have to be at least 2.

\subsection{Schellekens' List}
In \cite{Sch93} Schellekens classified the possible affine structures of ``meromorphic conformal field theories of central charge $24$''. In \cite{EMS15} his result is proved as a rigorous theorem on \voa{}s:
\begin{thm}[\cite{EMS15}, Theorem~6.4, \cite{Sch93}]\label{thm:schellekens}
Let $V$ be a \strathol{} \voa{} of central charge 24. Then one of the following holds:
\begin{enumerate}
\item $V_1=\{0\}$,
\item $V_1$ is the 24-dimensional abelian Lie algebra $\C^{24}$ (and $V$ is isomorphic to the \voa{} associated with the Leech lattice),
\item $V_1$ is one of 69 semisimple Lie algebras with a certain affine structure given in Table~1 of \cite{Sch93}.
\end{enumerate}
\end{thm}
We remark that entry 62 in Table~1 of \cite{Sch93} should read $E_{8,2}B_{8,1}$ (this typo is corrected in the arXiv-version of the paper).
\begin{proof}[Idea of Proof]
The modular invariance of the character of $V$ imposes strong constraints on the Lie algebra $V_1$ and its module $V_2$. In particular, $V_1$ is either trivial, abelian of dimension $24$ or semisimple \cite{DM04}. Suppose in the following that $V_1$ is semisimple.

Let $\g$ be a simple Lie algebra and $M$ a finite-dimensional $\g$-module with weight decomposition $M=\bigoplus_{\mu\in\Pi(M)}M_\mu$. We consider the power sums
\begin{equation*}
S_M^j(z):=\sum_{\mu\in\Pi(M)}\dim(M_\mu)\mu(z)^j
\end{equation*}
for $j\in\N$, which are polynomials in $z\in\hh$. This definition generalises straightforwardly to semisimple $\g$.

Modular invariance of the character of $V$ implies the power sum identity
\begin{equation*}
S_{V_1}^2(z)=\frac{\dim(V_1)-24}{12}\langle z,z\rangle,
\end{equation*}
which in turn implies \cite{DM04}
\begin{equation}\label{eq:dm}
\frac{h_i^\vee}{k_i}=\frac{\dim(V_1)-24}{24}
\end{equation}
for each simple component $\g_i$ of $V_1$ where $h_i^\vee$ is the dual Coxeter number of $\g_i$. Equation~\eqref{eq:dm} has 221 solutions (see \cite{EMS15}, Proposition~6.3, \cite{Sch93}).

Modular invariance implies further power sum identities on $V_2$, e.g.\
\begin{equation*}
S_{V_2}^2(z)=\left(32808-2\dim(V_1)\right)\langle z,z\rangle.
\end{equation*}
In \cite{EMS15}, Theorem~6.1, a list of these identities is presented. They are strong enough to constrain $V_1$ to be one of an explicit list of $69$ semisimple Lie algebras. Together with the trivial and abelian case for $V_1$ these give the 71 affine structures in Schellekens' list (see \cite{Sch93} and \cite{EMS15}, Theorem~6.4).
\end{proof}

Joint efforts \cite{FLM88,DGM90,Don93,DGM96,Lam11,LS12a,LS15a,Miy13b,SS16,EMS15,Moe16,LS16a,LS16,LL16} have finally led to:
\begin{thm}
All 71 affine structures in Schellekens' list are realised as $V_1$-structure of a \strathol{} \voa{} of central charge $24$.
\end{thm}

An actual classification of \voa{}s can be obtained if the following uniqueness conjecture is proved:
\begin{conj}\label{conj:unique}
Let $V$ be a \strathol{} \voa{} of central charge $24$. Then the Lie algebra structure of $V_1$ uniquely determines the \voa{} $V$ up to isomorphism.
\end{conj}

In this text we will prove the uniqueness conjecture for fifteen cases (fourteen of them new)\footnote{The uniqueness of case $B_8E_{8,2}$ (case 62) has already been proved in \cite{LS15a} using framed \voa{}s.} on Schellekens' list (see Theorem~\ref{thm:main}).

\subsection{Lattice \VOA{}s}
To keep the exposition self-contained we recall a few well-known facts about lattice \voa{}s \cite{Bor86,FLM88,Don93}.

Let $L$ be a positive-definite, even lattice, i.e.\ a free abelian group $L$ of finite rank~$r$ equipped with a positive-definite, symmetric bilinear form $\langle\cdot,\cdot\rangle\colon L\times L\to\Z$ such that the norm $\langle\alpha,\alpha\rangle/2$ is an integer for all $\alpha\in L$. We denote by $\Aut(L)=O(L)$ the group of automorphisms (or isometries) of the lattice $L$. The lattice \voa{} $V_L$ constructed from $L$ is simple, \strat{}, has central charge $c=r$ and carries a grading $V_L=\bigoplus_{\alpha\in L}V_L(\alpha)$ by $L$. If $L$ is unimodular, then $V_L$ is holomorphic.

The complexified lattice $\h:=L\otimes_\Z\C$ can be naturally identified with $\hh:=\{h(-1)\otimes\ee_0\,|\,h\in\h\}\subseteq(V_L)_1$, which is a Cartan subalgebra for the reductive Lie algebra $(V_L)_1$.

The construction of $V_L$ involves a choice of group $2$-cocycle $\eps\colon L\times L\to\{\pm 1\}$. An automorphism $\nu\in\Aut(L)$ together with a function $\eta\colon L\to\{\pm 1\}$ satisfying
\begin{equation*}
\eta(\alpha)\eta(\beta)/\eta(\alpha+\beta)=\eps(\alpha,\beta)/\eps(\nu\alpha,\nu\beta)
\end{equation*}
defines an automorphism $\hat{\nu}\in\Aut(V_L)$ (see, e.g. \cite{FLM88,Bor92}). We call $\hat{\nu}$ a \emph{standard lift} if the restriction of $\eta$ to the fixed-point sublattice $L^\nu\subseteq L$ is trivial. All standard lifts of $\nu$ are conjugate in $\Aut(V_L)$ (see \cite{EMS15}, Proposition~7.1). Let $\hat{\nu}$ be a standard lift of $\nu$ and suppose that $\nu$ has order $m$. If $m$ is odd or if $m$ is even and $\langle\alpha,\nu^{m/2}\alpha\rangle$ is even for all $\alpha\in L$, then the order of $\hat{\nu}$ is also $m$. Otherwise the order of $\hat{\nu}$ is $2m$, in which case we say $\nu$ exhibits order doubling.

There are precisely $24$ positive-definite, even, unimodular lattices of rank $24$, the Niemeier lattices. Each is determined up to isomorphism by its root lattice, and we write $N(Q)$ for the Niemeier lattice with non-zero root lattice $Q$. The one remaining Niemeier lattice $\Lambda$ with zero root lattice is called the Leech lattice.

\subsection{Inner Automorphisms}
We describe inner automorphisms of \voa{}s and the twisted modules associated with them. This summary is based on \cite{Li96}, Section~5. Let $V$ be a \voa{} of CFT-type. We define $K:=\langle\{\e^{v_0}\,|\,v\in V_1\}\rangle\subseteq\Aut(V)$ to be the \emph{inner automorphism group} of $V$. Note that since $V$ is of CFT-type, $v_0\omega=0$ for all $v\in V_1$ so that the inner automorphisms preserve the Virasoro vector $\omega$, which is included in our definition of \voa{} automorphism.

Let $v\in V_1$ such that the zero mode $v_0$ acts semisimply on $V$. Then the inner automorphism $\sigma_v:=\e^{-(2\pi\i)v_0}$ fulfils $\sigma_v^T=\id$ for some $T\in\Ns$ if and only if all eigenvalues of $v_0$ lie in $(1/T)\Z$.

Note that by definition of the Lie algebra structure on $V_1$, $\sigma_v$ restricts to the inner automorphism $\e^{-(2\pi\i)\ad_v}\in\Inn(V_1)\subseteq\Aut(V_1)$ of $V_1$.

Now, let $\sigma\in\Aut(V)$ be an automorphism of $V$ such that $\sigma v=v$. Then $\sigma$ and $\sigma_v$ commute.
%
We also assume that $L_1v=0$. If $V$ is self-contragredient and of CFT-type, then by results in \cite{Li94}, $L_1V_1=\{0\}$ and the assumption holds for any $v\in V_1$.

Suppose that $M=(M,Y_M)$ is a $\sigma$-twisted $V$-module. We consider Li's operator
\begin{equation*}
\Delta_v(z):=z^{v_0}\exp\left(\sum_{n=1}^\infty\frac{v_n}{-n}(-z)^{-n}\right)
\end{equation*}
and define $M^{(v)}=(M^{(v)},Y_{M^{(v)}})$ as $M^{(v)}:=M$ (as vector space) and
\begin{equation*}
Y_{M^{(v)}}(a,z):=Y_M(\Delta_v(z)a,z)
\end{equation*}
for all $a\in V$. Then $M^{(v)}$ is a $\sigma_v\sigma$-twisted $V$-module, which is irreducible if and only if $M$ is.\footnote{Note that we use the modern sign convention in the definition of twisted modules (see, e.g.\ \cite{DLM00} as opposed to \cite{Li96}). This is why we included a minus sign in the definition of $\sigma_v=\e^{-(2\pi\i)v_0}$.}

The $L_0$-grading of $M^{(v)}$ is shifted in comparison to $M$ via
\begin{equation*}
L_0^{M^{(v)}}=L_0^M+v_0+\frac{1}{2}(v_1v)_{-1}
\end{equation*}
where the $v_n$, $n\in\Z$, denote the modes of the unmodified module vertex operator $Y_M(\cdot,z)$.
If $V$ is self-contragredient, then we can rewrite the last term and obtain
\begin{equation}\label{eq:weights}
L_0^{M^{(v)}}=L_0^M+v_0+\frac{1}{2}\langle v,v\rangle\id_M
\end{equation}
using the normalised, symmetric, invariant bilinear form $\langle\cdot,\cdot\rangle$ on $V$.


\subsection{Kac's Theory}\label{sec:kac}
In a beautiful theory Kac classifies the finite-order automorphisms of finite-dimensional, simple Lie algebras and their fixed-point Lie subalgebras (see
\cite{Kac90}, Chapter~8). We briefly describe the results relevant for this work. Let $\g$ be a finite-dimensional, simple Lie algebra of type $X_n$. Then the conjugacy class of a finite-order automorphism of $\g$ is uniquely determined by the following data:
\begin{enumerate}
\item an affine Dynkin diagram of type $X_n^{(k)}$ where $k\in\{1,2,3\}$,
\item a sequence of non-negative, relatively prime integers $s_0,\ldots,s_l$ associated with the nodes of the affine Dynkin diagram where we consider two sequences identical if they can be transformed into each other by a diagram automorphism.
\end{enumerate}
In the case $k=1$ (untwisted affine Dynkin diagram) the automorphism of $\g$ is inner, and if $k=2,3$ (twisted affine Dynkin diagram), then the automorphism is outer, i.e. it projects to a non-trivial diagram automorphism of $X_n$ of order $k$. The order of the automorphism of $\g$ is given by $k\sum_{i=0}^la_is_i$ with the Kac labels\footnote{See \cite{Kac90}, Tables Aff~1, Aff~2 and Aff~3. For the untwisted affine Dynkin diagrams, for example, these are given by the ``marks'' described in Section~\ref{sec:la} complemented with $a_0=1$.} $a_i$.

Moreover, the fixed-point Lie subalgebra of $\g$ is isomorphic to the semisimple Lie algebra corresponding to the subdiagram of $X_n^{(k)}$ of those nodes with vanishing $s_i$ plus an abelian part of dimension equal to the number of nodes with non-vanishing $s_i$ minus 1.

It is even possible to state when two inner automorphisms are not just conjugate but conjugate under an inner automorphism. This is the case if and only if the associated integers $s_i$ can be transformed into one another by an element in a certain subgroup of the automorphism group of the untwisted affine Dynkin diagram.

Two conjugate outer automorphisms are always conjugate under an inner automorphism, except for the simple type $D_4$ (see Lemma~\ref{lem:innerouter}).

It is straightforward to extend Kac's theory to semisimple or reductive Lie algebras.

\section{Orbifold Construction}\label{sec:orbifold}
In \cite{EMS15} (see also \cite{Moe16}) we develop a cyclic orbifold theory for holomorphic \voa{}s. Here we give a short summary. Let $V$ be a \strathol{} \voa{} (of central charge $c\in8\Ns$) and $G=\langle\sigma\rangle$ a finite, cyclic group of automorphisms of $V$ of order $n\in\Ns$.

By \cite{DLM00} there is an up to isomorphism unique irreducible $\sigma^i$-twisted $V$-module, which we call $V(\sigma^i)$, for each $i\in\Z_n$. Moreover, there is a representation $\phi_i\colon G\to\Aut_\C(V(\sigma^i))$ of $G$ on the vector space $V(\sigma^i)$ such that
\begin{equation*}
\phi_i(\sigma)Y_{V(\sigma^i)}(v,x)\phi_i(\sigma)^{-1}=Y_{V(\sigma^i)}(\sigma v,x)
\end{equation*}
for all $i\in\Z_n$, $v\in V$. This representation is unique up to an $n$-th root of unity. Let $\xi_n$ be the $n$-th root of unity $\xi_n:=\e^{2\pi\i/n}$. We denote the eigenspace of $\phi_i(\sigma)$ in $V(\sigma^i)$ corresponding to the eigenvalue $\xi_n^j$ by $W^{(i,j)}$. On $V(\sigma^0)=V$ a possible choice for $\phi_0$ is given by $\phi_0(\sigma)=\sigma$.

The \fpvosa{} $V^\sigma=W^{(0,0)}$ is simple and \strat{} by \cite{DM97,Miy15,CM16} and has exactly $n^2$ irreducible modules, namely the $W^{(i,j)}$, $i,j\in\Z_n$ \cite{MT04}. We further show that the conformal weight of $V(\sigma)$ is in $(1/n^2)\Z$, and we define the type $n\{t\}$ of $\sigma$ where $t\in\Z_n$ such that $t=n^2\rho(V(\sigma))\pmod{n}$.

Let us assume in the following that $\sigma$ has type $n\{0\}$, i.e.\ that $\rho(V(\sigma))\in(1/n)\Z$. In this case it is possible to choose the representations $\phi_i$ such that the conformal weights obey
\begin{equation*}
\rho(W^{(i,j)})\in\frac{ij}{n}+\Z
\end{equation*}
and $V^\sigma$ has fusion rules
\begin{equation*}
W^{(i,j)}\boxtimes W^{(l,k)}\cong W^{(i+l,j+k)}
\end{equation*}
for all $i,j,k,l\in\Z_n$ (see \cite{EMS15}, Section~5), i.e.\ the fusion ring of $V^\sigma$ is the group ring $\C[\Z_n\times\Z_n]$. In particular, all $V^\sigma$-modules are simple currents.

We also show that the characters of the irreducible $V^\sigma$-modules
\begin{equation*}
\ch_{W^{(i,j)}}(\tau)=\tr_{W^{(i,j)}}q^{L_0-c/24},
\end{equation*}
$i,j\in\Z_n$, $q=\e^{2\pi\i\tau}$, $\tau$ a complex variable in the upper half-plane $\H$, are modular forms of weight~0 for $\Gamma(n)$ if $24\mid c$ and for $\Gamma(\lcm(3,n))$ otherwise. They form a vector-valued modular form for the Weil representation of $\SLZ$ (times a character, which is trivial for $24\mid c$).

In general, a simple \voa{} $V$ is said to satisfy the \emph{positivity assumption} if the conformal weight $\rho(W)>0$ for any irreducible $V$-module $W\not\cong V$ and $\rho(V)=0$.

Now, if $V^\sigma$ satisfies the positivity assumption, then the direct sum of $V^\sigma$-modules
\begin{equation}\label{eq:orbifold}
V^{\orb(\sigma)}:=\bigoplus_{i\in\Z_n}W^{(i,0)}
\end{equation}
admits the structure of a \strathol{} \voa{} of the same central charge as $V$ and is called \emph{orbifold} of $V$ with respect to $\sigma$ \cite{EMS15}. Note that $\bigoplus_{j\in\Z_n}W^{(0,j)}$ is just the old \voa{} $V$.

\subsection{Inverse Orbifold}\label{sec:invorb}
The inverse orbifold was first described in \cite{EMS15}. Suppose that the \strathol{} \voa{} $V^{\orb(\sigma)}$ is obtained as an orbifold as described above. Then via $\amgis v:=\xi_n^iv$ for $v\in W^{(i,0)}$ we define an automorphism $\amgis$ of $V^{\orb(\sigma)}$ of order~$n$, and the unique irreducible $\amgis^j$-twisted $V^{\orb(\sigma)}$-module is given by $V^{\orb(\sigma)}(\amgis^j)=\bigoplus_{i\in\Z_n}W^{(i,j)}$, $j\in\Z_n$ (see \cite{Moe16}, Theorem~4.9.6, for the proof). Then
\begin{equation*}
(V^{\orb(\sigma)})^{\orb(\amgis)}=\bigoplus_{j\in\Z_n}W^{(0,j)}=V,
\end{equation*}
i.e.\ orbifolding with $\amgis$ is inverse to orbifolding with $\sigma$.

\section{Dimension Formulae}\label{sec:dimform}

The modular properties of the characters of the irreducible $V^\sigma$-modules can be used to derive a formula for the dimension of the weight-one space $V^{\orb(\sigma)}_1$ if $c=24$ and if $n$ is such that the modular curve $X_0(n)=\Gamma_0(n)\backslash\H^*$ has genus zero, in which case $\Gamma_0(n)$ is called a genus zero group. (Here $\Gamma:=\SLZ$, $\Gamma_0(n)=\left\{\left(\begin{smallmatrix}a&b\\c&d\end{smallmatrix}\right)\in\Gamma\,|\,c=0\pmod{n}\right\}$ and $\H^*=\H\cup P$ with the complex upper half-plane $\H$ and the cusps $P=\Q\cup\{\infty\}$.) The group $\Gamma_0(n)$ has genus zero if and only if $n=1,2,3,4,5,6,7,8,9,10,12,13,16,18,25$. Then, in particular, the field of modular functions, i.e.\ the meromorphic modular forms of weight~0, for the congruence subgroup $\Gamma_0(n)$ is generated by a single meromorphic function $t_n$ called a \emph{Hauptmodul}. This means that every modular function for $\Gamma_0(n)$ is a rational function of this Hauptmodul, i.e.\ lies in $\C(t_n)$.

The main result of this section is the following theorem:
\begin{thm}[Dimension Formula]\label{thm:dimform}
Let $V$ be a \strathol{} \voa{} of central charge 24 and $\sigma$ an automorphism of $V$ of type $n\{0\}$, and assume that $V^\sigma$ satisfies the positivity assumption. Let $V^{\orb(\sigma)}=\bigoplus_{i\in\Z_n}W^{(i,0)}$ be the orbifold \voa{}~\eqref{eq:orbifold}. Suppose $\Gamma_0(n)$ has genus zero. Then
\begin{equation*}
\sum_{d\mid n}\frac{\phi((d,n/d))}{(d,n/d)}\left(24+\frac{n}{d}\dim(V_1^\sigma)-\dim(V_1^{\orb(\sigma^d)})\right)=24+R
\end{equation*}
holds where $\phi$ is Euler's totient function and
\begin{equation*}
R=\frac{24}{\phi(n)}\sum_{k=1}^{n-1}\sum_{\substack{i,j\in\Z_n\\ij=k\!\!\!\!\pmod{n}}}d_{i,j,k}\dim(W^{(i,j)}_{k/n})
\end{equation*}
where the coefficients $d_{i,j,k}$ are certain non-negative integers (given in Remark~\ref{rem:R}).
\end{thm}
Note that if $\sigma$ is of type $n\{0\}$, then for any $d\mid n$ the automorphism $\sigma^d$ has type $(n/d)\{0\}$ so that the orbifold \voa{} $V^{\orb(\sigma^d)}$ exists.
\begin{rem}\label{rem:sumcuspwidths}
The sum of the coefficients of $\dim(V^\sigma_1)$ on the left-hand side is the the sum of all cusp widths for $\Gamma_0(n)$, which is equal to the Dedekind psi function
\begin{equation*}
\psi(n)=\sum_{d\mid n}\frac{n}{d}\frac{\phi((d,n/d))}{(d,n/d)}=n\prod_{\substack{p\mid n\\p\text{ prime}}}(1+1/p)
\end{equation*}
and, since $\Gamma_0(n)$ contains the negative of the identity matrix, also equals $[\Gamma:\Gamma_0(n)]$.
\end{rem}

The dimension formulae for $n=2,3$ are given in \cite{Mon94} 
(see also \cite{LS16b}, Theorem~4.8). For $n=2,3,5,7,13$ a proof is given in \cite{Moe16}, Proposition~4.10.4. A complete proof of the theorem will be given in Section~\ref{sec:dimproof}.

The theorem simplifies if the conformal weights satisfy $\rho(V(\sigma^i))\geq 1$ for all $i\in\Z_n\setminus\{0\}$ since then $R=0$.
\begin{cor}\label{cor:dimform}
In the situation of the above theorem let us assume that the conformal weights of the twisted modules obey $\rho(V(\sigma^i))\geq 1$ for all $i\in\Z_n\setminus\{0\}$. Then
\begin{equation*}
\dim(V^{\orb(\sigma)}_1) = 24 + \sum_{d\mid n}c_d \dim(V_1^{\sigma^d})
\end{equation*}
with
\begin{equation*}
c_d=\frac{\lambda(d)}{d}\,\frac{\phi((d,n/d))}{(d,n/d)}\,\psi(n/d)
\end{equation*}
where
\begin{equation*}
\lambda(d)=\prod_{\substack{p\text{ prime}\\p\mid d}}(-p).
\end{equation*}
Explicitly, the coefficients are given by:
\begin{equation*}
\renewcommand{\arraystretch}{1.2}
\begin{tabular}{r|l}
 $n$& $c_d$\\\hline
 $2$& $c_1=3$, $c_2=-1$\\
 $3$& $c_1=4$, $c_3=-1$\\
 $4$& $c_1=6$, $c_2=-\frac{3}{2}$, $c_4=-\frac{1}{2}$\\
 $5$& $c_1=6$, $c_5=-1$\\
 $6$& $c_1=12$, $c_2=-4$, $c_3=-3$, $c_6=1$\\
 $7$& $c_1=8$, $c_7=-1$\\
 $8$& $c_1=12$, $c_2=-3$, $c_4=-\frac{3}{4}$, $c_8=-\frac{1}{4}$\\
 $9$& $c_1=12$, $c_3=-\frac{8}{3}$, $c_9=-\frac{1}{3}$\\
$10$& $c_1=18$, $c_2=-6$, $c_5=-3$, $c_{10}=1$\\
$12$& $c_1=24$, $c_2=-6$, $c_3=-6$, $c_4=-2$, $c_6=\frac{3}{2}$, $c_{12}=\frac{1}{2}$\\
$13$& $c_1=14$, $c_{13}=-1$\\
$16$& $c_1=24$, $c_2=-6$, $c_4=-\frac{3}{2}$, $c_8=-\frac{3}{8}$, $c_{16}=-\frac{1}{8}$\\
$18$& $c_1=36$, $c_2=-12$, $c_3=-8$, $c_6=\frac{8}{3}$, $c_9=-1$, $c_{18}=\frac{1}{3}$\\
$25$& $c_1=30$, $c_5=-\frac{24}{5}$, $c_{25}=-\frac{1}{5}$
\end{tabular}
\renewcommand{\arraystretch}{1}
\end{equation*}
\end{cor}
\begin{proof}
Recursion of the dimension formula in Theorem~\ref{thm:dimform} gives the coefficients $c_d$. Then it is a simple check that they satisfy the stated formula.  
\end{proof}
Observe that $c_1=\psi(n)$ and $\sum_{d\mid n}c_d=n$.

\begin{rem}\label{rem:R}
We describe the coefficients $d_{i,j,k}$ in the term $R$ in Theorem~\ref{thm:dimform}. They appear for $i,j,k\in\{1,\ldots,n-1\}$ with $k=ij\pmod{n}$. If $n$ is prime, then
\begin{equation*}
d_{i,j,k}=\sigma(n-k)
\end{equation*}
where $\sigma$ is the sum-of-divisors function $\sigma(m)=\sum_{d|m}d$.

For any $n$, if $(i,j,n)=1$, then
\begin{equation*}
d_{i,j,k}=\sum_{\substack{d\mid(n-k)\\(d,n)=1}}\frac{n-k}{d},
\end{equation*}
describing in particular all coefficients in the cases where $n$ is the square of a prime. This only leaves a few further coefficients, which we list in the following:
\begin{equation*}
\renewcommand{\arraystretch}{1.1}
\begin{minipage}[t]{0.5\textwidth}
\centering
\begin{tabular}[t]{c|c|c|c}
$n$ & $(i,j,n)$ & $ij$ & $d_{i,j,k}$\\\hline
\multirow{ 3}{*}{$\s6$}&\multirow{2}{*}{$2$}&$\s2\pmod{\s6}$&$\s5$\\
                       &                    &$\s4\pmod{\s6}$&$\s1$\\\cline{2-4}
                       &                $3$ &$\s3\pmod{\s6}$&$\s2$\\\hline
\multirow{ 2}{*}{$\s8$}&\multirow{2}{*}{$2$}& $\s4\pmod{16}$&$\s2$\\
                       &                    &  $12\pmod{16}$&$\s6$\\\hline
\multirow{ 5}{*}{ $10$}&\multirow{4}{*}{$2$}& $\s2\pmod{10}$& $13$\\
                       &                    & $\s4\pmod{10}$&$\s4$\\
                       &                    & $\s6\pmod{10}$&$\s5$\\
                       &                    & $\s8\pmod{10}$&$\s1$\\\cline{2-4}
                       &                $5$ & $\s5\pmod{10}$&$\s4$\\\hline
\multirow{ 9}{*}{ $12$}&\multirow{4}{*}{$2$}& $\s4\pmod{24}$&$\s4$\\
                       &                    & $\s8\pmod{24}$&$\s2$\\
                       &                    &  $16\pmod{24}$& $12$\\
                       &                    &  $20\pmod{24}$&$\s6$\\\cline{2-4}
                       &\multirow{3}{*}{$3$}& $\s3\pmod{12}$& $14$\\
                       &                    & $\s6\pmod{12}$&$\s4$\\
                       &                    & $\s9\pmod{12}$&$\s2$\\\cline{2-4}
                       &\multirow{2}{*}{$4$}& $\s4\pmod{12}$&$\s8$\\
                       &                    & $\s8\pmod{12}$&$\s2$
\end{tabular}
\end{minipage}
\begin{minipage}[t]{0.5\textwidth}
\centering
\begin{tabular}[t]{c|c|c|c}
$n$ & $(i,j,n)$ & $ij$ & $d_{i,j,k}$\\\hline
\multirow{ 6}{*}{ $16$}&\multirow{6}{*}{$2$}&$\s4\pmod{32}$&$\s8$\\
                       &                    &$\s8\pmod{32}$&$\s4$\\
                       &                    & $12\pmod{32}$&$\s2$\\
                       &                    & $20\pmod{32}$& $24$\\
                       &                    & $24\pmod{32}$& $12$\\
                       &                    & $28\pmod{32}$&$\s6$\\\hline
\multirow{12}{*}{ $18$}&\multirow{8}{*}{$2$}&$\s2\pmod{18}$& $29$\\
                       &                    &$\s4\pmod{18}$&$\s8$\\
                       &                    &$\s6\pmod{18}$& $15$\\
                       &                    &$\s8\pmod{18}$&$\s6$\\
                       &                    & $10\pmod{18}$& $13$\\
                       &                    & $12\pmod{18}$&$\s3$\\
                       &                    & $14\pmod{18}$&$\s5$\\
                       &                    & $16\pmod{18}$&$\s1$\\\cline{2-4}
                       &\multirow{3}{*}{$3$}&$\s9\pmod{54}$&$\s6$\\
                       &                    & $27\pmod{54}$&$\s6$\\
                       &                    & $45\pmod{54}$& $15$\\\cline{2-4}
                       &                $9$ &$\s9\pmod{18}$&$\s6$
\end{tabular}
\end{minipage}
\renewcommand{\arraystretch}{1}
\end{equation*}
Note that the coefficients obey the symmetries $d_{i,j,k}=d_{j,i,k}=d_{n-i,n-j,k}$.
\end{rem}

\subsection{Proof of Dimension Formula}\label{sec:dimproof}
In this section we prove Theorem~\ref{thm:dimform}. The argument is based on the theory of (classical) modular forms and vector-valued modular forms for the Weil representation. For the necessary background we refer the reader to \cite{Miy06,Sch09,EMS15}. In this text, unless otherwise stated, modular forms are assumed to be holomorphic on the upper half-plane but may have poles at the cusps.

Let $V$ be a \strathol{} \voa{} of central charge 24 and $\sigma$ an automorphism of $V$ of type $n\{0\}$, and assume that $V^\sigma$ satisfies the positivity assumption. Let $V^{\orb(\sigma)}=\bigoplus_{i\in\Z_n}W^{(i,0)}$ be the orbifold \voa{}~\eqref{eq:orbifold}.

Then, as shown in \cite{EMS15}, the characters $\ch_{W^{(i,j)}}(\tau)$ form a vector-valued modular form of weight~0 for the Weil representation of the group $D=\Z_n\times\Z_n$ with quadratic form $q((i,j))=ij/n+\Z$, i.e.\ they transform as
\begin{align*} 
\ch_{W^{(i,j)}}(T.\tau)&=\e^{(2\pi\i)ij/n}\ch_{W^{(i,j)}}(\tau),\\
\ch_{W^{(i,j)}}(S.\tau)&=\frac{1}{n}\sum_{k,l\in\Z_n}\e^{-(2\pi\i)(il+jk)/n}\ch_{W^{(k,l)}}(\tau)
\end{align*}
where $S=\left(\begin{smallmatrix}0&-1\\1&0\end{smallmatrix}\right)$ and $T=\left(\begin{smallmatrix}1&1\\0&1\end{smallmatrix}\right)$ are the standard generators of $\Gamma=\SLZ$. This implies that
\begin{equation*}
\ch_{W^{(0,0)}}(\tau)=\ch_{V^\sigma}(\tau)
\end{equation*}
is a modular form of weight 0 for $\Gamma_0(n)$ by Proposition~4.5 in \cite{Sch09}.

The cusps $P=\Q\cup\{\infty\}$ fall into $\sum_{c\mid n}\phi((c,n/c))$ classes modulo $\Gamma_0(n)$ (see \cite{Miy06}, Theorem~4.2.7). The cusp classes can be represented by the rational numbers $a/c$ with $(a,c)=1$ where $c$ ranges over the positive divisors of $n$ and $a$ over a complete set of representatives of $\Z_{(c,n/c)}^\times$. We denote by $t_s:=(n/c)/(c,n/c)$ the width of the cusp $s=a/c$.

For a cusp $s\in P=\Q\cup\{\infty\}$ we define
\begin{equation*}
F_s(\tau):=\sum_{\substack{M\in\Gamma_0(n)\backslash\Gamma\\M.\infty=s}}\ch_{W^{(0,0)}}(M.\tau).
\end{equation*}
\begin{lem}
Let $s=a/c\in\Q$ be a cusp of $\Gamma_0(n)$ with $(a,c)=1$ and $c\mid n$. Choose $d$ such that $ad=1\pmod{c}$. Then
\begin{equation*} 
F_s(\tau)=\frac{1}{(c,n/c)}\sum_{\substack{i,j\in\Z_{n/c}\\ij=0\!\!\!\!\pmod{t_s}}}\e^{-(2\pi\i)(dij/t_s)/(c,n/c)}\ch_{W^{(ci,cj)}}(\tau).
\end{equation*}
\end{lem}
Note that $ij/t_s$ is well-defined modulo $(c,n/c)$.
\begin{proof}
Choose a matrix $M=\left(\begin{smallmatrix}a&b\\c&d\end{smallmatrix}\right)\in\Gamma$. Then the cosets in $\Gamma_0(n)\backslash\Gamma$ sending $\infty$ to $s$ are given by $MT^j$ where $j$ ranges over a complete set of residues modulo $t_s$. The statement now follows from 
\begin{equation}\label{eq:chcusp}
\ch_{W^{(0,0)}}(M.\tau)=\frac{c}{n}\sum_{i,j\in\Z_{n/c}}\e^{-(2\pi\i)dcij/n}\ch_{W^{(ci,cj)}}(\tau)
\end{equation}
(see \cite{Sch09}, Theorem~4.7 and \cite{Sch15}, Section~2).
\end{proof}

The orbifold \voa{} for $\sigma^d$ where $d\mid n$ is given by
\begin{equation*}
V^{\orb(\sigma^d)}=\bigoplus_{\substack{j\in\Z_{n/d}\\k\in \Z_d}}W^{(dj,kn/d)}.
\end{equation*}
Together with the above lemma this implies the following relation between the characters of the orbifolds $V^{\orb(\sigma^d)}$ and the functions $F_s$. 
\begin{prop}\label{prop:stt}
The following formula holds:
\begin{equation*}
\sum_{s\in\Gamma_0(n)\backslash P} F_s(\tau)=\sum_{d\mid n}\frac{\phi((d,n/d))}{(d,n/d)}\ch_{V^{\orb(\sigma^d)}}(\tau).
\end{equation*}
\end{prop}

So far the results do not depend on the genus of $\Gamma_0(n)$.
\begin{lem}\label{lem:fs}
Suppose $\Gamma_0(n)$ has genus zero, and let $s\in P$ be a cusp. If $s$ is not equivalent to $\infty$ under $\Gamma_0(n)$, then there is a modular function $f_s$ for $\Gamma_0(n)$ which is holomorphic and non-zero on $\H^*$ except for $\ord_s(f_s)=-1/t_s$ and $\ord_\infty(f_s)=1$.

If $s$ is equivalent to $\infty$, there is a modular function $f_s$ for $\Gamma_0(n)$ which is holomorphic and non-zero on $\H^*$ except for $\ord_s(f_s)=-1=-1/t_s$ and $\ord_0(f_s)=1/n$.
\end{lem}
The functions $f_s$ are unique up to a non-zero factor.
\begin{proof}
By Abel's Theorem (see, e.g.\ \cite{Mir95}, Chapter~VIII) the modular curve $X_0(n)=\Gamma_0(n)\backslash\H^*$ has genus zero if and only if every divisor of degree zero is a principal divisor. Hence, any divisor $D$ supported only at the cusps of $\Gamma_0(n)$ and satisfying
\begin{equation*}
\sum_{s\in\Gamma_0(n)\backslash P}\ord_s(D)t_s=0
\end{equation*}
is the divisor of a meromorphic function on $X_0(n)$ that is holomorphic and non-zero away from the cusps.
\end{proof}

\begin{lem}
Suppose $\Gamma_0(n)$ has genus zero. Then the character $\ch_{V^\sigma}(\tau)$ can be written as 
\begin{equation*}
\ch_{V^\sigma}(\tau)=\sum_{s\in\Gamma_0(n)\backslash P}g_s(f_s(\tau))+c\end{equation*}
with complex polynomials $g_s$ and a complex constant $c$. The polynomials $g_s$ are unique up to constants (once we fix a choice of the $f_s$).
\end{lem}
\begin{proof}
The character $\ch_{V^\sigma}(\tau)$ is a modular form of weight $0$, which is holomorphic on the upper half-plane and has possible poles at the cusps. Since $f_s$ has its only pole at the cusp $s$, there are complex polynomials $g_s$ such that 
\begin{equation*}
\ch_{V^\sigma}(\tau)-\sum_{s\in\Gamma_0(n)\backslash P}g_s(f_s(\tau))
\end{equation*}
is holomorphic on the upper half-plane and at all cusps and hence constant.
\end{proof}
For definiteness we choose the polynomials $g_s$ to have zero constant term. We can make the statement of the above theorem even more precise. Noting that \eqref{eq:chcusp} implies that $\ch_{V^\sigma}(\tau)$ has a pole of order $-1$ at every cusp $s$ of $\Gamma_0(n)$ and comparing this with the pole of order $-1/t_s$ of the function $f_s$ at the cusp $s$ we see that
\begin{equation*}
\ch_{V^\sigma}(\tau)=c+\sum_{s\in\Gamma_0(n)\backslash P}\sum_{i=1}^{t_s}c_{s,i}f_s(\tau)^i
\end{equation*}
for certain constants $c_{s,i}\in\C$, $s\in\Gamma_0(n)\backslash P$, $i\in\{1,\ldots,t_s\}$, i.e.\ the polynomial
\begin{equation*}
g_s(x)=\sum_{i=1}^{t_s}c_{s,i}x^i
\end{equation*}
is of degree $t_s$ with zero constant term. The constant $c$ is given by $\dim(V_1^\sigma)-[f_\infty](0)$ where $[f_\infty](0)$ denotes the constant coefficient in the $q$-expansion of $f_\infty$ since all $f_s$ vanish at $\infty$ except for $f_\infty$.

Suppose we have determined the polynomials $g_s$ or equivalently the constants $c_{s,i}$. Then we can calculate the expansions of $\ch_{V^\sigma}(\tau)$ at the different cusps, and using Proposition~\ref{prop:stt} we can compute
\begin{equation*}
\sum_{d\mid n}\frac{\phi((d,n/d))}{(d,n/d)}\dim(V_1^{\orb(\sigma^d)})=\sum_{s\in\Gamma_0(n)\backslash P}[F_s](0).
\end{equation*}

There are several ways to determine the coefficients of the polynomials $g_s$. One possibility is as follows. Note that
\begin{equation*}
\ch_{V^\sigma}(\tau)=\ch_{W^{(0,0)}}(\tau)=\frac{1}{n}\sum_{j\in\Z_n}T(\vac,0,j,\tau)
\end{equation*}
in terms of twisted characters (see, e.g.\ \cite{EMS15} for their definition and properties in the orbifold context) so that 
\begin{align*}
\ch_{V^\sigma}(M.\tau)&=\frac{1}{n}\sum_{j\in\Z_n}T(\vac,(0,j)M,\tau)=\frac{1}{n}\sum_{j\in\Z_n}T(\vac,cj,dj,\tau)\\
&=\frac{1}{n}\sum_{j\in\Z_n}\tr_{V(\sigma^{cj})}\phi_{cj}(\sigma^{dj})q^{L_0-1}=\frac{1}{n}\sum_{j,k\in\Z_n}\xi_n^{djk}\tr_{W^{(cj,k)}}q^{L_0-1}
\end{align*}
for $M=\left(\begin{smallmatrix}a&b\\c&d\end{smallmatrix}\right)\in\Gamma$. Hence we can determine the coefficients of the polynomials $g_s$ in the decomposition
\begin{equation*}
\ch_{V^\sigma}(\tau)=\sum_{s\in\Gamma_0(n)\backslash P}g_s(f_s(\tau))+c
\end{equation*}
by expanding both sides at the cusp $s$ using the above formula for $\ch_{V^\sigma}(M.\tau)$ with $M.\infty=s$ and comparing the $t_s$ many singular $q$-coefficients.

We then obtain
\begin{equation*}
\sum_{d\mid n}\frac{\phi((d,n/d))}{(d,n/d)}\dim(V_1^{\orb(\sigma^d)})=\sum_{s\in\Gamma_0(n)\backslash P}[F_s](0)
\end{equation*}
explicitly as some constant plus a multiple of $\dim(V_1^\sigma)$ plus some linear combination of the dimensions $\dim(W^{(i,j)}_{k/n})$ with $k/n<1$.

By definition, $F_s$ is a sum of $|\{M\in\Gamma_0(n)\backslash\Gamma\,|\,M.\infty=s\}|=t_s$ many terms, each containing $\dim(V_1^\sigma)$ once, so that every $F_s$ contributes $t_s\dim(V_1^\sigma)$, and hence this dimension appears with a total multiplicity of $\sum_{s\in\Gamma_0(n)\backslash P}t_s=[\Gamma:\Gamma_0(n)]$ (see Remark~\ref{rem:sumcuspwidths}).

Finally, we describe our choice of the functions $f_s$ for some of the cases considered in Theorem~\ref{thm:dimform}.

First, let $n=2,3,5,7,13$. Then a Hauptmodul for $\Gamma_0(n)$ is given by the eta product
\begin{equation*}
t_n(\tau)=\left(\frac{\eta(\tau)}{\eta(n\tau)}\right)^{24/(n-1)}=q^{-1}-\frac{24}{n-1}+\ldots
\end{equation*}
where $\eta(\tau)$ is the Dedekind eta function. The Hauptmodul is holomorphic and non-zero on $\H$. There are two classes of cusps of $\Gamma_0(n)$, represented by $1/n$ ($=\infty$) and $1/1$ ($=0$). We define the functions
\begin{align*}
f_{1/n}(\tau)&:=t_n(\tau)=\left(\frac{\eta(\tau)}{\eta(n\tau)}\right)^{24/(n-1)},\\
f_{1/1}(\tau)&:=\frac{1}{t_n(\tau)}=\left(\frac{\eta(n\tau)}{\eta(\tau)}\right)^{24/(n-1)}
\end{align*}
for the cusps $1/n$ and $1/1$ of $\Gamma_0(n)$, respectively. These functions are holomorphic and non-zero on the upper half-plane $\H$ and have exactly the poles and zeroes at the cusps demanded in Lemma~\ref{lem:fs}.

Now suppose $n=4$. A Hauptmodul for $\Gamma_0(4)$ is given by
\begin{equation*}
t_4(\tau)=\frac{\eta(\tau)^8}{\eta(4\tau)^8}=q^{-1}-8+\ldots
\end{equation*}
There are three classes of cusps of $\Gamma_0(4)$, represented by $1/4$ ($=\infty$), $1/2$ and $1/1$ ($=0$). We define the functions
\begin{align*}
f_{1/4}(\tau)&:=t_4(\tau)=\frac{\eta(\tau)^8}{\eta(4\tau)^8},\\
f_{1/2}(\tau)&:=\frac{1}{t_4(\tau)+16}=\frac{\eta(\tau)^8\eta(4\tau)^{16}}{\eta(2\tau)^{24}},\\
f_{1/1}(\tau)&:=\frac{1}{t_4(\tau)}=\frac{\eta(4\tau)^8}{\eta(\tau)^8}
\end{align*}
for the cusps $1/4$, $1/2$ and $1/1$ of $\Gamma_0(4)$, respectively. These functions are holomorphic and non-zero on the upper half-plane $\H$ and have exactly the poles and zeroes at the cusps demanded in Lemma~\ref{lem:fs}.

Let $n=6$. A Hauptmodul for $\Gamma_0(6)$ is given by
\begin{equation*}
t_6(\tau)=\frac{\eta(\tau)^5\eta(3\tau)}{\eta(2\tau)\eta(6\tau)^5}=q^{-1}-5+\ldots
\end{equation*}
There are four classes of cusps of $\Gamma_0(6)$, represented by $1/6$ ($=\infty$), $1/3$, $1/2$ and $1/1$ ($=0$). We define the functions
\begin{align*}
f_{1/6}(\tau)&:=t_6(\tau)=\frac{\eta(\tau)^5\eta(3\tau)}{\eta(2\tau)\eta(6\tau)^5},\\
f_{1/3}(\tau)&:=\frac{1}{t_6(\tau)+8}=\frac{\eta(\tau)^3\eta(6\tau)^9}{\eta(2\tau)^3\eta(3\tau)^9},\\
f_{1/2}(\tau)&:=\frac{1}{t_6(\tau)+9}=\frac{\eta(\tau)^4\eta(6\tau)^8}{\eta(2\tau)^8\eta(3\tau)^4},\\
f_{1/1}(\tau)&:=\frac{1}{t_6(\tau)}=\frac{\eta(2\tau)\eta(6\tau)^5}{\eta(\tau)^5\eta(3\tau)}
\end{align*}
for the cusps $1/6$, $1/3$, $1/2$ and $1/1$ of $\Gamma_0(6)$, respectively. These functions are holomorphic and non-zero on the upper half-plane $\H$ and have exactly the poles and zeroes at the cusps demanded in Lemma~\ref{lem:fs}.

Let $n=8$. A Hauptmodul for $\Gamma_0(8)$ is given by
\begin{equation*}
t_8(\tau)=\frac{\eta(\tau)^4\eta(4\tau)^2}{\eta(2\tau)^2\eta(8\tau)^4}=q^{-1}-4+\ldots
\end{equation*}
There are four classes of cusps of $\Gamma_0(8)$, represented by $1/8$ ($=\infty$), $1/4$, $1/2$ and $1/1$ ($=0$). We define the functions
\begin{align*}
f_{1/8}(\tau)&:=t_8(\tau)=\frac{\eta(\tau)^4\eta(4\tau)^2}{\eta(2\tau)^2\eta(8\tau)^4},\\
f_{1/4}(\tau)&:=\frac{1}{t_8(\tau)+4}=\frac{\eta(2\tau)^4\eta(8\tau)^8}{\eta(4\tau)^{12}},\\
f_{1/2}(\tau)&:=\frac{1}{t_8(\tau)+8}=\frac{\eta(\tau)^4\eta(4\tau)^2\eta(8\tau)^4}{\eta(2\tau)^{10}},\\
f_{1/1}(\tau)&:=\frac{1}{t_8(\tau)}=\frac{\eta(2\tau)^2\eta(8\tau)^4}{\eta(\tau)^4\eta(4\tau)^2}
\end{align*}
for the cusps $1/8$, $1/4$, $1/2$ and $1/1$ of $\Gamma_0(8)$, respectively. These functions are holomorphic on the upper half-plane $\H$ and have exactly the poles and zeroes at the cusps demanded in Lemma~\ref{lem:fs}.

In the remaining cases one can easily construct the functions $f_s$ using the expressions of the Hauptmoduln for $\Gamma_0(n)$ as eta products given in \cite{CN79}, Section~11, Table~3.

\section{Uniqueness Procedure}\label{sec:uniqueness}
The following idea is due to Lam and Shimakura (see \cite{LS16b}, Theorem~5.2) and makes use of the inverse orbifold construction developed in \cite{EMS15}.

\subsection{General Idea}
Let $U$ be a \strathol{} \voa{} of central charge $c$ and $g$ an automorphism of $U$ of type $n\{0\}$ for some $n\in\Ns$. Assume that $U^g$ satisfies the positivity assumption. Then by \cite{EMS15} we can apply the $\Z_n$-orbifold construction to $U$ and $g$ and obtain the orbifolded \voa{} $U^{\orb(g)}$, which is also \strathol{} and of central charge $c$.

On the other hand, let $V$ be a \strathol{} \voa{} of central charge $c$ with weight-one Lie algebra $V_1\cong U^{\orb(g)}_1$.
Suppose that there exists an element $v\in V_1$ such that $v_0$ acts semisimply on $V$, the inner automorphism $\sigma_v:=\e^{-(2\pi\i)v_0}$ has type $n\{0\}$ on $V$ and $V^{\sigma_v}$ satisfies the positivity assumption. Then we can apply the $\Z_n$-orbifold construction to $V$ and $\sigma_v$ and obtain the \strathol{} \voa{} $V^{\orb(\sigma_v)}$ of central charge $c$. Assume that $V^{\orb(\sigma_v)}\cong U$.

By the inverse orbifold construction \cite{EMS15} there exists an automorphism $\amgis\in\Aut(V^{\orb(\sigma_v)})$ of type $n\{0\}$ with $(V^{\orb(\sigma_v)})^{\amgis}=V^{\sigma_v}$ such that the orbifold construction with $V^{\orb(\sigma_v)}$ and $\amgis$ gives back $V$, i.e.\ $(V^{\orb(\sigma_v)})^{\orb(\amgis)}=V$. Finally, assume that under the isomorphism $V^{\orb(\sigma_v)}\cong U$ the inverse orbifold automorphism $\amgis\in\Aut(V^{\orb(\sigma_v)})$ of $\sigma_v$ is conjugate to $g\in\Aut(U)$ (and in particular $V_1^{\sigma_v}\cong U^g_1$). Then the orbifold of $V^{\orb(\sigma_v)}$ by $\amgis$, which is $V$, is isomorphic to the orbifold $U^{\orb(g)}$ of $U$ by $g$, i.e.\ $V\cong U^{\orb(g)}$.

\subsection{Concrete Approach}\label{sec:approach}
In order to prove the uniqueness conjecture (Conjecture~\ref{conj:unique}) for some cases on Schellekens' list we specialise to the following situation. 

\subsubsection*{Part 1: Construction}
Let $U:=V_L$ be the lattice \voa{} associated with some Niemeier lattice $L$. Then $V_L$ is \strathol{} and of central charge 24. Let $\nu\in\Aut(L)$ be an automorphism of the lattice $L$ and $\hat\nu\in\Aut(V_L)$ a standard lift of $\nu$. We denote by $n\in\Ns$ the order of $\hat\nu$, which is $\ord(\nu)$ or $2\ord(\nu)$ (possible order doubling). Suppose that the conformal weight $\rho(V_L(\hat\nu))=0\pmod{1/n}$, i.e.\ $\hat\nu$ is of type $n\{0\}$. Then the \fpvosa{} $V_L^{\hat\nu}$ satisfies the positivity assumption. We apply the orbifold construction \cite{EMS15} and obtain the orbifolded \voa{} $V_L^{\orb(\hat\nu)}$, which is again \strathol{} and of central charge 24. Then we compute the dimension of the Lie algebra $(V_L^{\orb(\hat\nu)})_1$, which has to be one of the 71 Lie algebras on Schellekens' list. Suppose we are able to determine the Lie algebra structure of $(V_L^{\orb(\hat\nu)})_1$ up to isomorphism.

\subsubsection*{Part 2: Inner Automorphism}
We want to show that $V_L^{\orb(\hat\nu)}$ is up to isomorphism the only \strathol{} \voa{} of central charge 24 with weight-one Lie algebra isomorphic to $(V_L^{\orb(\hat\nu)})_1$. To this end, consider another \strathol{} \voa{} $V$ of central charge 24 with $V_1\cong (V_L^{\orb(\hat\nu)})_1$.

The Lie algebra $V_1\cong (V_L^{\orb(\hat\nu)})_1$ is reductive or more precisely semisimple, abelian or zero. Assume that $V_1$ is semisimple. We fix a Cartan subalgebra $\hh$ of $V_1$. Let $h\in\hh$ such that $h_0$ acts semisimply on $V$.\footnote{It is shown in \cite{DM04b} and \cite{Mas14}, Section~3.3, that every $V$-module is a completely reducible $V_1$-module.
Then if $v\in V_1$ is a semisimple element, the zero mode $v_0$ acts semisimply on every $V$-module (and in particular on $V$ itself).
}
The action of $h_0$ on $V_1$ is given by $\ad_h$. In particular, $h_0x=[h,x]=\alpha(h)x$ for $x$ in the root space of $V_1$ associated with the root $\alpha\in\Phi\subseteq\hh^*$.

Let us consider the inner automorphism $\sigma_h:=\e^{-(2\pi\i)h_0}$ on $V$ and assume that it has type $n\{0\}$, that $V^{\sigma_h}$ satisfies the positivity assumption and that $V^{\sigma_h}_1\cong (V_L^{\hat\nu})_1$. Then we can apply the orbifold construction and obtain the \strathol{} \voa{} $V^{\orb(\sigma_h)}$ of central charge 24. Assume that we can compute the dimension $d:=\dim(V^{\orb(\sigma_h)}_1)$. More specifically, we assume that $\rho(V(\sigma_h^i))\geq 1$ for all $i\in\Z_n\setminus\{0\}$ so that we can apply the dimension formula in Corollary~\ref{cor:dimform}. Then suppose that $d=\dim((V_L)_1)$.

We know that $(V_L^{\hat\nu})_1$ is a fixed-point Lie subalgebra of $(V_L)_1$ under an automorphism of order dividing $n$, namely under $\hat\nu|_{(V_L)_1}$. Then, so is $V^{\sigma_h}_1\cong (V_L^{\hat\nu})_1$. Suppose that $(V_L)_1$ is the only Lie algebra on Schellekens' list in dimension $d$ with $(V_L^{\hat\nu})_1$ as a possible fixed-point Lie subalgebra under an automorphism of order dividing $n$ (the possible fixed-point Lie subalgebras can be determined using Kac's theory as described in Section~\ref{sec:kac}). Then we know that $V^{\orb(\sigma_h)}_1\cong (V_L)_1$ since $V^{\orb(\sigma_h)}_1$ has dimension $d$ by definition and has $V^{\sigma_h}_1\cong (V_L^{\hat\nu})_1$ as fixed-point Lie subalgebra (under the inverse orbifold automorphism $\amgis$ restricted to $V^{\orb(\sigma_h)}_1$, see below). Then, finally $V^{\orb(\sigma_h)}\cong V_L$ since it is shown in \cite{DM04b} that a \strathol{} \voa{} of central charge 24 with a Lie algebra of rank 24 is already isomorphic to a lattice \voa{}.

\subsubsection*{Part 3: Unique Conjugacy Class}
Let $\amgis\in\Aut(V^{\orb(\sigma_h)})$ be the inverse orbifold automorphism such that $(V^{\orb(\sigma_h)})^{\orb(\amgis)}\cong V$. Assume that we can show that $\amgis$ is conjugate to $\hat\nu\in\Aut(V_L)$ under the isomorphism $V^{\orb(\sigma_h)}\cong V_L$. More specifically we will show that there is a unique conjugacy class of automorphisms $g$ of $V_L$ of order~$n$ with fixed-point Lie subalgebra $(V_L^g)_1$ isomorphic to $V^{\sigma_h}_1\cong (V_L^{\hat\nu})_1$ and possibly the condition that the conformal weight $\rho(V_L(g))=1$.

Knowing that $\amgis$ and $\hat\nu$ are conjugate we conclude, as described above, that $V\cong V_L^{\orb(\hat\nu)}$. This then proves that the \voa{} $V_L^{\orb(\hat\nu)}$ is the up to isomorphism unique \strathol{} \voa{} of central charge 24 with weight-one space isomorphic as Lie algebra to $(V_L^{\orb(\hat\nu)})_1$.

\section{Part 1: Construction}\label{sec:part1}
The construction of orbifolds $V_L^{\orb(\hat\nu)}$ associated with a holomorphic lattice \voa{} $V_L$ and an automorphism $\hat\nu\in\Aut(V_L)$ of type $n\{0\}$ obtained as a standard lift of a lattice automorphism $\nu\in\Aut(L)$ is described in \cite{EMS15}, Section~8 and \cite{Moe16}, Sections 5.6, 6.2 and 6.3. Applying these methods we obtain orbifold constructions of the following \strathol{} \voa{}s of central charge 24:
\begin{prop}[Part 1]\label{prop:part1}
For each entry of the following table there exists an automorphism $\nu$ of the Niemeier lattice $L$ of order~$n$ and cycle shape $s$ such that the orbifold \voa{} $V_L^{\orb(\hat\nu)}$ of $V_L$ by a standard lift $\hat{\nu}$ (also of order~$n$) has the indicated weight-one Lie algebra $(V_L^{\orb(\hat{\nu})})_1$.
\begin{equation*}
\renewcommand{\arraystretch}{1.2}
\begin{tabular}{rlrlllc}
No.\ &$L$&$n$&$s$&$(V_L^{\hat\nu})_1$&$(V_L^{\orb(\hat\nu)})_1$& \cite{Sch93} \\\hline
 $(1)$ & $N(A_9^2D_6)$     & $2$ & $1^{-8}2^{16}$       & $A_5C_5D_5\C$         & $A_5C_5E_{6,2}$     & $44$\\
 $(2)$ & $N(A_7^2D_5^2)$   & $2$ & $1^{-8}2^{16}$       & $A_1^2A_3A_7B_2^2$    & $A_3A_{7,2}C_3^2$   & $33$\\
 $(3)$ & $N(A_8^3)$        & $2$ & $1^{-8}2^{16}$       & $A_8B_4$              & $A_{8,2}F_{4,2}$    & $36$\\
 $(4)$ & $N(E_8^3)$        & $2$ & $1^{-8}2^{16}$       & $D_8E_8$              & $B_8E_{8,2}$        & $62$\\
 $(5)$ & $N(A_5^4D_4)$     & $2$ & $1^{-8}2^{16}$       & $A_1A_2^2A_3A_5B_2\C$ & $A_2^2A_{5,2}^2B_2$ & $26$\\
 $(6)$ & $N(E_6^4)$        & $2$ & $1^{-8}2^{16}$       & $C_4^2F_4^2$          & $C_8F_4^2$          & $52$\\
 $(7)$ & $N(A_4^6)$        & $2$ & $1^{-8}2^{16}$       & $A_4^2B_2^2$          & $A_{4,2}^2C_{4,2}$  & $22$\\
 $(8)$ & $N(A_2^{12})$     & $2$ & $1^{-8}2^{16}$       & $A_1^4A_2^4$          & $A_{2,2}^4D_{4,4}$  & $13$\\
 $(9)$ & $N(A_{11}D_7E_6)$ & $2$ & $1^{-8}2^{16}$       & $A_1B_5D_6F_4$        & $B_5E_{7,2}F_4$     & $53$\\
$(10)$ & $N(A_{11}D_7E_6)$ & $2$ & $1^{-8}2^{16}$       & $B_2B_4C_4C_6$        & $B_4C_6^2$          & $48$\\
$(11)$ & $N(A_4^6)$        & $5$ & $1^{-1}5^5$          & $A_4\C^4$             & $A_{4,5}^2$         &$\s9$\\
$(12\mathrm{a})$& $N(D_{10}E_7^2)$&$4$&$1^22^{-9}4^{10}$& $A_1A_4A_7B_3\C$      & $A_4A_{9,2}B_3$     & $40$\\
$(12\mathrm{b})$& $N(D_{10}E_7^2)$&$4$&$2^{-4}4^8$      & $A_1A_4A_7B_3\C$      & $A_4A_{9,2}B_3$     & $40$\\
$(13)$ & $N(D_{10}E_7^2)$  & $4$ & $2^{-4}4^8$          & $A_7B_2B_6\C$         & $B_6C_{10}$         & $56$\\
$(14)$ & $N(E_6^4)$        & $6$ & $1^32^{-3}3^{-9}6^9$ & $A_1A_2C_4\C$         & $A_1C_{5,3}G_{2,2}$ & $21$\\
$(15)$ & $N(D_4^6)$        & $8$ & $4^{-2}8^4$          & $A_3\C^7$             & $A_{1,2}A_{3,4}^3$  &$\s7$
\end{tabular}
\renewcommand{\arraystretch}{1}
\end{equation*}
\end{prop}
Note that the automorphisms $\hat\nu$ in cases (12a) and (12b) are in fact conjugate in $\Aut(V_L)$ as is shown in Proposition~\ref{prop:part3} while the lattice automorphisms $\nu$ they are lifted from cannot be conjugate in $\Aut(L)$ because they do not have the same cycle shape.
\begin{proof}[Sketch of Proof]
For cases (1), (2) and (15) we let $\nu$ be the automorphism, unique up to conjugacy, specified by the following conditions:
\begin{enumerate}
\item[(1)] There are exactly two conjugacy classes in $\Aut(N(A_9^2D_6))$ of cycle shape $1^{-8}2^{16}$ and length $113\,400$.
For exactly one of them no order doubling occurs. This is the one we choose as $\nu$.
\item[(2)] There are exactly two conjugacy classes in $\Aut(N(A_7^2D_5^2))$ of cycle shape $1^{-8}2^{16}$ and length $403\,200$.
For both of them no order doubling occurs. We choose as $\nu$ the only one with $(V_L^{\hat\nu})_1\cong A_1^2A_3A_7B_2^2$.
\item[(15)] There are exactly four conjugacy classes in $\Aut(N(D_4^6))$ of cycle shape $4^{-2}8^4$ and length $26\,418\,075\,402\,240$.
Two of them exhibit no order doubling. We choose as $\nu$ the only one of those two with $(V_L^{\hat\nu})_1\cong A_3\C^7$.
\end{enumerate}
For the remaining cases we let $\nu$ be a representative of the unique conjugacy class
\begin{enumerate}
\item[(3)] in $\Aut(N(A_8^3))$ of cycle shape $1^{-8}2^{16}$ and length $1\,088\,640$,
\item[(4)] in $\Aut(N(E_8^3))$ of cycle shape $1^{-8}2^{16}$ and length $2\,090\,188\,800$,
\item[(5)] in $\Aut(N(A_5^4D_4))$ of cycle shape $1^{-8}2^{16}$ and length $518\,400$,
\item[(6)] in $\Aut(N(E_6^4))$ of cycle shape $1^{-8}2^{16}$ and length $12\,150$,
\item[(7)] in $\Aut(N(A_4^6))$ of cycle shape $1^{-8}2^{16}$ and length $216\,000$,
\item[(8)] in $\Aut(N(A_2^{12}))$ of cycle shape $1^{-8}2^{16}$ and length $641\,520$,
\item[(9)] in $\Aut(N(A_{11}D_7E_6))$ of cycle shape $1^{-8}2^{16}$ and length $1\,575$,
\item[(10)] in $\Aut(N(A_{11}D_7E_6))$ of cycle shape $1^{-8}2^{16}$ and length $218\,295$,
\item[(11)] in $\Aut(N(A_4^6))$ of cycle shape $1^{-1}5^5$ and length $119\,439\,360\,000$,
\item[(12a)] in $\Aut(N(D_{10}E_7^2))$ of cycle shape $1^22^{-9}4^{10}$ and length $877\,879\,296\,000$,
\item[(12b)] in $\Aut(N(D_{10}E_7^2))$ of cycle shape $2^{-4}4^8$ and length $1\,316\,818\,944\,000$,
\item[(13)] in $\Aut(N(D_{10}E_7^2))$ of cycle shape $2^{-4}4^8$ and length $18\,289\,152\,000$,
\item[(14)] in $\Aut(N(E_6^4))$ of cycle shape $1^32^{-3}3^{-9}6^9$ and length $1\,719\,926\,784\,000$,
\end{enumerate}
respectively. For each case (1) to (15) we let $\hat\nu$ be a standard lift of $\nu$. No order doubling occurs, and the automorphism $\hat\nu$ of $V_L$ is of type $n\{0\}$.

We note that the conformal weights of the twisted modules satisfy $\rho(V(\hat{\nu}^i))\geq 1$ for all $i\in\Z_n\setminus\{0\}$. Hence the dimension of the weight-one Lie algebra $(V_L^{\orb(\hat\nu)})_1$ can be computed using Corollary~\ref{cor:dimform}. (Alternatively, this dimension can be obtained by the method described in Section~8 of \cite{EMS15}.)

For each case (1) to (15) we know $\dim((V_L^{\orb(\hat\nu)})_1)$, and we know that $(V_L^{\orb(\hat\nu)})_1$ contains $(V_L^{\hat\nu})_1$ as the fixed-point Lie subalgebra under an automorphism of order dividing $n$, namely under the inverse orbifold automorphism with respect to $\hat\nu$. This is enough to uniquely determine the Lie algebra structure of $(V_L^{\orb(\hat\nu)})_1$ using Kac's theory (see Section~\ref{sec:kac}) and Schellekens' list. Actually, for the case (11) we need some additional arguments, but this case is described in detail in \cite{EMS15}.
\end{proof}


\section{Part 2: Inner Automorphism}\label{sec:part2}

In the following let $V$ be a \strathol{} \voa{} of central charge 24 with semisimple weight-one Lie algebra $V_1$. Let $\hh$ be a Cartan subalgebra of $V_1$. For each of the fifteen affine structures appearing on the right-hand side of the table in Proposition~\ref{prop:part1} we fix a semisimple element $h\in\hh$ (so that $h_0$ acts semisimply on $V$) as in the following table such that $\sigma_h=\e^{-(2\pi\i)h_0}$ has order~$n$:
\begin{equation*}\label{page:table}
\renewcommand{\arraystretch}{1.4}
\setlength{\tabcolsep}{.9\tabcolsep}
\begin{tabular}{rlllrcl}
No.\ &$V_1$& $h$ & ord.\ on $V_1$ & $n$ & $\langle h,h\rangle$ & $V^{\sigma_h}_1$ \\\hline
 $(1)$ & $A_5C_5E_{6,2}$     & $(0,0,\frac{\Lambda_1^\vee+\Lambda_6^\vee}{2})$                                   & $(1,1,2)$     &$2$&$2$& $A_5C_5D_5\C$\\
 $(2)$ & $A_3A_{7,2}C_3^2$   & $(0,0,\frac{\Lambda_2^\vee}{2},\frac{\Lambda_2^\vee}{2})$                         & $(1,1,2,2)$   &$2$&$2$& $A_1^2A_3A_7B_2^2$\\
 $(3)$ & $A_{8,2}F_{4,2}$    & $(0,\frac{\Lambda_4^\vee}{2})$                                                    & $(1,2)$       &$2$&$2$& $A_8B_4$\\
 $(4)$ & $B_8E_{8,2}$        & $(\frac{\Lambda_8^\vee}{2},0)$                                                    & $(2,1)$       &$2$&$2$& $D_8E_8$\\
 $(5)$ & $A_2^2A_{5,2}^2B_2$ & $(0,0,\frac{\Lambda_2^\vee+\Lambda_4^\vee}{2},0,0)$                               & $(1,1,2,1,1)$ &$2$&$2$& $A_1A_2^2A_3A_5B_2\C$\\
 $(6)$ & $C_8F_4^2$          & $(\frac{\Lambda_4^\vee}{2},0,0)$                                                  & $(2,1,1)$     &$2$&$2$& $C_4^2F_4^2$\\
 $(7)$ & $A_{4,2}^2C_{4,2}$  & $(0,0,\frac{\Lambda_2^\vee}{2})$                                                  & $(1,1,2)$     &$2$&$2$& $A_4^2B_2^2$\\
 $(8)$ & $A_{2,2}^4D_{4,4}$  & $(0,0,0,0,\frac{\Lambda_2^\vee}{2})$                                              & $(1,1,1,1,2)$ &$2$&$2$& $A_1^4A_2^4$\\
 $(9)$ & $B_5E_{7,2}F_4$     & $(0,\frac{\Lambda_6^\vee}{2},0)$                                                  & $(1,2,1)$     &$2$&$2$& $A_1B_5D_6F_4$\\
$(10)$ & $B_4C_6^2$          & $(0,0,\frac{\Lambda_4^\vee}{2})$                                                  & $(1,1,2)$     &$2$&$2$& $B_2B_4C_4C_6$\\
$(11)$ & $A_{4,5}^2$         & $(0,\frac{\Lambda_1^\vee+\Lambda_2^\vee+\Lambda_3^\vee+\Lambda_4^\vee}{5})$       & $(1,5)$       &$5$&$2$& $A_4\C^4$\\
$(12)$ & $A_4A_{9,2}B_3$     & $(0,\frac{\Lambda_1^\vee+3\Lambda_3^\vee}{4},0)$                                  & $(1,4,1)$     &$4$&$3$& $A_1A_4A_7B_3\C$\\
$(13)$ & $B_6C_{10}$         & $(0,\frac{\Lambda_2^\vee}{4}+\frac{\Lambda_{10}^\vee}{2})$                        & $(1,4)$       &$4$&$2$& $A_7B_2B_6\C$\\
$(14)$ & $A_1C_{5,3}G_{2,2}$ & $(0,\frac{\Lambda_4^\vee}{6}+\frac{2\Lambda^\vee_5}{3},\frac{\Lambda_1^\vee}{3})$ & $(1,6,3)$     &$6$&$8$& $A_1A_2C_4\C$\\
$(15)$ & $A_{1,2}A_{3,4}^3$  & $(\frac{\Lambda_1^\vee}{2},x,x,0)$                                                & $(2,8,8,1)$   &$8$&$2$& $A_3\C^7$
\end{tabular}
\renewcommand{\arraystretch}{1}
\end{equation*}
with $x=\frac{\Lambda_1^\vee+\Lambda_2^\vee+3\Lambda_3^\vee}{8}$. We give $h$ in terms of the fundamental coweights (as labelled in \cite{Hum73}, for example).
The fixed-point Lie subalgebra $V^{\sigma_h}_1$ can be easily read off from the definitions of $h$, using Kac's theory (see Section~\ref{sec:kac}), and we observe that $V^{\sigma_h}_1\cong (V_L^{\hat\nu})_1$ with $V_L^{\hat\nu}$ from the previous section. Of course, $h$ is chosen precisely such that this is the case.

The order $n$ of $\sigma_h$ is determined as follows: The order of $\sigma_h$ restricted to the semisimple Lie algebra $V_1$ is clearly a divisor of $n$. On the other hand, if $h\in(1/k)Q^\vee$ for some $k\in\Ns$, then $\langle h,\lambda\rangle\in(1/k)\Z$ for any weight $\lambda$ in the weight lattice $P$ and hence $\sigma_h$ has order dividing $k$ on the $V_1$-module $V$. In the above fifteen cases, it turns out that the upper and lower bounds coincide, i.e. $\ord(\sigma_h|_{V_1})=n$ and $h\in(1/n)Q^\vee$, so that $\ord(\sigma_h)=n$.

Since $\sigma_h$ has order~$n$, the eigenvalues of $h_0$ lie in $(1/n)\Z$ as does the value of $\langle h,h\rangle/2$. The eigenvalues of $L_0$ on $V$ lie in $\Z$. From formula~\eqref{eq:weights} for the weight grading of $V(\sigma_h)$ we conclude that the weight grading and hence the conformal weight of $V(\sigma_h)$ is in $(1/n)\Z$. Hence $\sigma_h$ is of type $n\{0\}$. One can check that $\rho(V(\sigma_h^i))>0$ for all $i\in\Z_n\setminus\{0\}$ so that $V^{\sigma_h}$ satisfies the positivity assumption. Indeed, we shall show in Lemma~\ref{lem:confgeqone} below that $\rho(V(\sigma_h^i))\geq1$ holds for each of the above cases. So we can apply the orbifold construction (see Section~\ref{sec:orbifold}) and obtain the \strathol{} \voa{} $V^{\orb(\sigma_h)}$ of central charge $c=24$.

The following section is aimed at the computation of a lower bound for the conformal weight of the twisted modules $V(\sigma_h^i)$, $i\in\Z_n\setminus\{0\}$. Since we do not know the complete \voa{} structure of $V$ at this point, the approach will be rather indirect. The idea for the combinatorial argument in the next section is from \cite{LS16b}.

\subsection{Lower Bound for Conformal Weight}
We continue in the above setting, considering the decomposition $V_1=\g_1\oplus\ldots\oplus\g_r$ of $V_1$ into simple ideals. As explained in Section~\ref{sec:affine}, $V$ decomposes into a direct sum~\eqref{eq:Vdecomp} of irreducible $\langle V_1\rangle$-modules of non-negative integer conformal weight. While the conformal weight of $\langle V_1\rangle$ is $0$, the conformal weights (given by formula~\eqref{eq:confsum}) of all other irreducible $\langle V_1\rangle$-modules appearing in the decomposition are at least $2$.

Now we study the conformal weight of the unique irreducible $\sigma_h$-twisted $V$-module $V(\sigma_h)=V^{(h)}$. It is clear that $V(\sigma_h)$ decomposes into the direct sum of irreducible $\sigma_h$-twisted $\langle V_1\rangle$-modules $M^{(h)}$ for those modules $M$ appearing in the decomposition of $V$ into irreducible $\langle V_1\rangle$-modules.

\begin{lem}[\cite{LS16b}, Lemma~2.7]\label{lem:conftwisted}
Let $V$ be a \strathol{} \voa{} of central charge 24 with $V_1=\g_1\oplus\ldots\oplus\g_r$ semisimple and $h=(h_1,\ldots,h_r)$ a semisimple element in the Cartan subalgebra $\hh$ of $V_1$ such that $\sigma_h$ has finite order. In addition we assume that
\begin{equation}\label{eq:rootcond}
\alpha(h_i)\geq-1
\end{equation}
for all roots $\alpha\in\Phi(\g_i)$ and $i=1,\ldots,r$. Let $M\cong L_{\hat\g_1}(k_1,\lambda_1)\otimes\ldots\otimes L_{\hat\g_r}(k_r,\lambda_r)$ be an irreducible $\langle V_1\rangle$-submodule of $V$. Then the conformal weight of the irreducible $\sigma_h$-twisted $\langle V_1\rangle$-module $M^{(h)}$ is given by
\begin{equation*}
\rho(M^{(h)})=\rho(M)+\sum_{i=1}^r\min\{\mu(h_i)\,|\,\mu\in\Pi(\lambda_i,\g_i)\}+\frac{\langle h,h\rangle}{2}
\end{equation*}
where $\Pi(\lambda_i,\g_i)$ is the set of all weights of the irreducible $\g_i$-module $L_{\g_i}(\lambda_i)$ with highest weight $\lambda_i$.
\end{lem}
Note that condition~\eqref{eq:rootcond} is essential for this formula. Also, since $\sigma_h$ has finite order, $h$ is a $\Q$-linear combination of the fundamental coweights.
\begin{rem}
Let $\g$ be a simple Lie algebra, $\lambda\in P_+$ a dominant integral weight and $h$ some element in the Cartan subalgebra of $\g$. If $h$ is an $\R$-linear combination of the fundamental coweights, the minimum $\min\{\mu(h)\,|\,\mu\in\Pi(\lambda,\g)\}$ can be computed as follows. Clearly, the minimum does not change if we modify $h$ by an element in the Weyl group $W$ of $\g$. Note that there is an element $w$ in the Weyl group such that $h_-:=wh$ is a $\R_{\leq0}$-linear combination of the fundamental coweights, and hence $\min\{\mu(h)\,|\,\mu\in\Pi(\lambda,\g)\}=\lambda(h_-)$.
\end{rem}

It is easy to check:
\begin{lem}\label{lem:conftwisted2}
For each case (1) to (15) in the table on p.~\pageref{page:table} the vector $h$ fulfils the assumptions of Lemma~\ref{lem:conftwisted}. In particular, for every $i$, $\alpha(h_i)\geq-1$ for all roots $\alpha\in\Phi(\g_i)$ of the simple component $\g_i$ of $V_1$.
\end{lem}
The two lemmas allow us to compute the conformal weights of the possible irreducible $\sigma_h$-twisted $\langle V_1\rangle$-modules $M^{(h)}$ appearing in the decomposition of $V(\sigma_h)=V^{(h)}$, which gives a lower bound for $\rho(V(\sigma_h))$.
\begin{lem}\label{lem:confgeqone0}
For each case (1) to (15) in the table on p.~\pageref{page:table} the bound
\begin{equation*}
\rho(V(\sigma_h))\geq 1
\end{equation*}
holds.
\end{lem}
\begin{proof}
It suffices to show that $\rho(M^{(h)})\geq 1$ for each irreducible component $M$ of the decomposition~\eqref{eq:Vdecomp} of $V$. Consider the irreducible $\langle V_1\rangle$-module
\begin{equation*}
M=L_{\hat\g_1}(k_1,\lambda_1)\otimes\ldots\otimes L_{\hat\g_r}(k_r,\lambda_r).
\end{equation*}
If $(\lambda_1,\ldots,\lambda_r)=(0,\ldots,0)$, then by Lemmas \ref{lem:conftwisted} and \ref{lem:conftwisted2} the conformal weight of $M^{(h)}$ is equal to $\langle h,h\rangle/2\geq1$.

Now we consider $(\lambda_1,\ldots,\lambda_r)\neq(0,\ldots,0)$, in which case $\rho(M)\in\Z_{\geq2}$. Using Lemmas \ref{lem:conftwisted} and \ref{lem:conftwisted2} it is straightforward to compute the conformal weight $\rho(M^{(h)})$ for any irreducible $\langle V_1\rangle$-module $M$. In all the cases except (11) and (15) we find that $\rho(M^{(h)})\geq 1$ for all irreducible $\langle V_1\rangle$-modules $M$ satisfying $\rho(M)\in\Z_{\geq2}$.

Cases (11) and (15) require more involved arguments. We begin with case~(11), where $V_1=A_{4,5}^2$ and $h=(0,(\Lambda_1^\vee+\Lambda_2^\vee+\Lambda_3^\vee+\Lambda_4^\vee)/5)$. Here, there are eleven pairs $(\lambda_1,\lambda_2)$ that cannot be excluded by the simple argument above. More precisely, if the irreducible $\langle V_1\rangle$-module $M$ has highest weight among
\begin{equation*}
([0,0,0,0],[0,2,2,0]),~([0,0,0,0],[1,0,2,2]),~([0,0,0,0],[2,2,0,1]),
\end{equation*}
then $\rho(M^{(h)})=3/5$, and if $M$ has highest weight among
\begin{align*}
&([0,0,0,1],[0,1,2,1]),~([0,0,0,1],[1,2,1,0]),\\
&([0,0,0,1],[2,0,1,2]),~([0,0,0,1],[2,1,0,2]),\\
&([1,0,0,0],[0,1,2,1]),~([1,0,0,0],[1,2,1,0]),\\
&([1,0,0,0],[2,0,1,2]),~([1,0,0,0],[2,1,0,2]),
\end{align*}
then $\rho(M^{(h)})=4/5$. Here we have written the weight $\sum_{i=1}^lm_i\Lambda_i$ as $[m_1,\ldots,m_l]$ in terms of the fundamental weights. In all these cases $\rho(M)=2$.

The \voa{} $\langle V_1\rangle\cong(L_{A_4}(5,0))^{\otimes2}$ has $146$ irreducible modules $M$ of conformal weight $\rho(M)=2$. The identities listed in \cite{EMS15}, Theorem~6.1, are equivalent to a linear system of equations on the subset of the $146$ corresponding multiplicities $m_{(\lambda_1,\lambda_2)}$ in the decomposition~\eqref{eq:Vdecomp}. We seek solutions in non-negative integers to this system. An application of the simplex algorithm shows that all multiplicities vanish except possibly those of the weights
\begin{align*}
&([0,0,0,0],[0,0,0,5]),~([0,1,1,0],[1,2,0,0]),\\
&([2,0,1,0],[2,0,1,0]),~([0,1,1,0],[2,0,0,2])
\end{align*}
and of their images under the action of the group $\Z_2^2$ generated by exchange of the two copies of $A_4$ and the involution $\Lambda_i\mapsto\Lambda_{4-i}$, $i=1,\ldots,4$. Hence, none of the above eleven problematic modules can occur in the decomposition~\eqref{eq:Vdecomp}.

In case~(15) with $V_1=A_{1,2}A_{3,4}^3$ and $h=(\frac{\Lambda_1^\vee}{2},\frac{\Lambda_1^\vee+\Lambda_2^\vee+3\Lambda_3^\vee}{8},\frac{\Lambda_1^\vee+\Lambda_2^\vee+3\Lambda_3^\vee}{8},0)$ we proceed similarly. Here, there are seventeen tuples $(\lambda_1,\ldots,\lambda_4)$ that cannot be excluded by the simple combinatorial argument, namely
\begin{align*}
&([0],[2,1,0],[2,1,0],[0,0,0]),\\
&([1],[1,0,1],[3,0,1],[0,0,0]),
([1],[2,0,0],[2,0,2],[0,0,0]),\\
&([1],[2,0,2],[2,0,0],[0,0,0]),
([1],[3,0,1],[1,0,1],[0,0,0]),\\
&([2],[1,0,1],[2,1,0],[0,0,0]),
([2],[1,1,1],[2,0,0],[0,0,0]),\\
&([2],[2,0,0],[1,1,1],[0,0,0]),
([2],[2,1,0],[1,0,1],[0,0,0])
\end{align*}
with $\rho(M^{(h)})=3/4$ and
\begin{align*}
&([0],[1,0,1],[4,0,0],[0,0,0]),
([0],[4,0,0],[1,0,1],[0,0,0]),\\
&([1],[0,1,0],[4,0,0],[0,0,0]),
([1],[3,0,1],[4,0,0],[0,0,0]),\\
&([1],[4,0,0],[0,1,0],[0,0,0]),
([1],[4,0,0],[3,0,1],[0,0,0]),\\
&([2],[2,1,0],[4,0,0],[0,0,0]),
([2],[4,0,0],[2,1,0],[0,0,0])
\end{align*}
with $\rho(M^{(h)})=7/8$. In all these cases $\rho(M)\in\{2,3\}$.

Again we show that none of these modules occur in the decomposition of $V$ into irreducible modules for $\langle V_1\rangle\cong L_{A_1}(2,0)\otimes(L_{A_3}(4,0))^{\otimes 3}$. Indeed, as described in the proof sketch of Theorem~\ref{thm:schellekens}, the modular invariance implies a large system of linear equations on the multiplicities $m_{(\lambda_1,\ldots,\lambda_4)}$ occurring in the decomposition~\eqref{eq:Vdecomp}. In the present case the system is underdetermined, but somewhat surprisingly the sum $m_{\lambda^{(1)}}+m_{\lambda^{(2)}}+m_{\lambda^{(3)}}$ of the multiplicities of the modules with highest weights
\begin{align*}
\lambda^{(1)}&=([0],[0,0,0],[0,0,0],[0,4,0]),\\
\lambda^{(2)}&=([0],[0,0,0],[0,4,0],[0,0,0]),\\
\lambda^{(3)}&=([0],[0,4,0],[0,0,0],[0,0,0])
\end{align*}
is forced to be $3$.

We observe that each of these $\langle V_1\rangle$-modules is a simple current of order~$2$. Recall that a \emph{simple current} is an irreducible module $M$ with the property that fusion with $M$ acts by permutation on the set of all irreducible modules.

In general, if a simple current $M$ occurs in the decomposition~\eqref{eq:Vdecomp}, then the multiplicities are constant along orbits under fusion by $M$. In particular, the multiplicity of $M$ itself equals that of $\langle V_1\rangle\cong L_{\hat\g_1}(k_1,0)\otimes\ldots\otimes L_{\hat\g_r}(k_r,0)$, which is $1$. It follows in the present case that $m_{\lambda^{(1)}}=m_{\lambda^{(2)}}=m_{\lambda^{(3)}}=1$. There is therefore an action of $\Z_2^3$ on the set of irreducible $\langle V_1\rangle$-modules by permutations, under which the multiplicities of~\eqref{eq:Vdecomp} are constant.

We treat these equalities as extra equations to be added to the system. In the reduced system we observe that the sum $m_{\lambda^{(4)}}+m_{\lambda^{(5)}}+m_{\lambda^{(6)}}$ of the multiplicities of the modules with highest weights 
\begin{align*}
\lambda^{(4)}&=([2],[0,0,0],[0,0,0],[0,0,4]),\\
\lambda^{(5)}&=([2],[0,0,0],[0,0,4],[0,0,0]),\\
\lambda^{(6)}&=([2],[0,0,4],[0,0,0],[0,0,0])
\end{align*}
is forced to be $3$. Again, these modules are simple currents, so we conclude that $m_{\lambda^{(4)}}=m_{\lambda^{(5)}}=m_{\lambda^{(6)}}=1$, and we supplement the system with equations coming from the action by fusion of these modules.

The system, thus supplemented, has a unique solution in positive integers. The corresponding decomposition~\eqref{eq:Vdecomp} is exactly the one given in \cite{Sch93}.

We can easily verify that every irreducible $\langle V_1\rangle$-module $M$ occurring in \eqref{eq:Vdecomp} with positive multiplicity satisfies the property $\rho(M^{(h)})\geq1$, i.e.\ none of the above seventeen problematic modules occurs.
\end{proof}

We would like to extend the result of Lemma~\ref{lem:confgeqone0} to the twisted $V$-modules $V(\sigma_h^i)$ for all $i\in\Z_n\setminus\{0\}$. For $n=2$ we are done, and for $n=3$ it suffices to recall that the contragredient module $(V(\sigma_h))'\cong V(\sigma_h^{-1})=V(\sigma_h^2)$ has the same grading as $V(\sigma_h)$.

For $n>3$ we have to consider modules twisted by $\sigma_h^i=\sigma_{ih}$ for $i>1$. In general, condition~\eqref{eq:rootcond} is not satisfied by $ih=(ih_1,\ldots,ih_r)$, so Lemma~\ref{lem:conftwisted} cannot be applied directly. To circumvent this difficulty we observe that $V$ is graded by the weight lattice relative to the action of $\hh$. Hence for any element $k$ of the coroot lattice $Q^\vee$ the automorphisms $\sigma_h$ and $\sigma_{h+k}$ of $V$ coincide. Furthermore it is always possible to find a representative $i[h]$ of $ih+Q^\vee$ satisfying the condition~\eqref{eq:rootcond}, i.e.\
\begin{equation*}
\alpha(i[h]_j)\geq-1
\end{equation*}
for all roots $\alpha\in\Phi(\g_j)$ and all $j=1,\ldots,r$. (The element $i[h]$ is in general not unique.)

\begin{lem}\label{lem:confgeqone}
For the cases (1) to (15) in the table on p.~\pageref{page:table}
\begin{equation*}
\rho(V(\sigma_h^i))\geq 1
\end{equation*}
for all $i\in\Z_n\setminus\{0\}$.
\end{lem}
\begin{proof}
We proceed as described above, choosing the following representatives $i[h]$ where for $n/2<i<n$ we can of course take $i[h]=-((n-i)[h])$:
\begin{enumerate}
\item[(11)] For $V_1\cong A_{4,5}^2$
\begin{align*}
h&=(0,\frac{\Lambda_1^\vee+\Lambda_2^\vee+\Lambda_3^\vee+\Lambda_4^\vee}{5}),\\
2[h]&=(0,\frac{-3\Lambda_1^\vee+2\Lambda_2^\vee+2\Lambda_3^\vee-3\Lambda_4^\vee}{5}).
\end{align*}
\item[(12)] For $V_1\cong A_4A_{9,2}B_3$
\begin{align*}
h&=(0,\frac{\Lambda_1^\vee+3\Lambda_3^\vee}{4},0),\\
2[h]&=(0,\frac{-\Lambda_1^\vee-\Lambda_3^\vee}{2}+\Lambda_4^\vee-\Lambda_5^\vee+\Lambda_6^\vee-\Lambda_7^\vee+\Lambda_8^\vee-\Lambda_9^\vee,0).
\end{align*}
\item[(13)] For $V_1\cong B_6C_{10}$
\begin{align*}
h&=(0,\frac{\Lambda_2^\vee}{4}+\frac{\Lambda_{10}^\vee}{2}),\\
2[h]&=(0,\frac{\Lambda_2^\vee}{2}-\Lambda_3^\vee+\Lambda_4^\vee-\Lambda_5^\vee+\Lambda_6^\vee-\Lambda_7^\vee+\Lambda_8^\vee-\Lambda_9^\vee+\Lambda_{10}^\vee).
\end{align*}
\item[(14)] For $V_1\cong A_1C_{5,3}G_{2,2}$
\begin{align*}
h&=(0,\frac{\Lambda_4^\vee}{6}+\frac{2\Lambda^\vee_5}{3},\frac{\Lambda^\vee_1}{3}),\\
2[h]&=(0,\frac{\Lambda_4^\vee}{3}-\frac{2\Lambda^\vee_5}{3},\frac{2\Lambda^\vee_1}{3}-\Lambda_2^\vee),\\
3[h]&=(0,-\Lambda_1^\vee+\Lambda_2^\vee-\Lambda_3^\vee+\frac{\Lambda_4^\vee}{2},0).
\end{align*}
\item[(15)] For $V_1\cong A_{1,2}A_{3,4}^3$
\begin{align*}
h&=(\frac{\Lambda_1^\vee}{2},\frac{\Lambda_1^\vee+\Lambda_2^\vee+3\Lambda_3^\vee}{8},\frac{\Lambda_1^\vee+\Lambda_2^\vee+3\Lambda_3^\vee}{8},0),\\
2[h]&=(\Lambda_1^\vee,\frac{-3\Lambda_1^\vee+\Lambda_2^\vee-\Lambda_3^\vee}{4},\frac{-3\Lambda_1^\vee+\Lambda_2^\vee-\Lambda_3^\vee}{4},0),\\
3[h]&=(-\frac{\Lambda_1^\vee}{2},\frac{-5\Lambda_1^\vee+3\Lambda_2^\vee+\Lambda_3^\vee}{8},\frac{-5\Lambda_1^\vee+3\Lambda_2^\vee+\Lambda_3^\vee}{8},0),\\
4[h]&=(0,\frac{\Lambda_1^\vee-\Lambda_2^\vee-\Lambda_3^\vee}{2},\frac{\Lambda_1^\vee-\Lambda_2^\vee-\Lambda_3^\vee}{2},0).
\end{align*}
\end{enumerate}
Then we repeat the proof of Lemma~\ref{lem:confgeqone0} to show that $\rho(V(\sigma_{i[h]}))\geq 1$ for all $i\in\Z_n\setminus\{0\}$. For all cases except for (11) and (15) this works without complication. In cases (11) and (15) we make use of the same information on the decomposition of $V$ as a $\langle V_1\rangle$-module that we used in the proof of Lemma~\ref{lem:confgeqone0}.
\end{proof}

\subsection{Dimension Formula}
With Lemma~\ref{lem:confgeqone} in hand we may apply the dimension formula in Corollary~\ref{cor:dimform} to compute $d=\dim(V_1^{\orb(\sigma_h)})$ in our fifteen cases. This information, along with Kac's theory (see Section~\ref{sec:kac}), will permit us to determine the Lie algebra structure of $V_1^{\orb(\sigma_h)}$.
\begin{prop}[Part 2]\label{prop:part2}
Let $V$ and $\sigma_h$ be as in the table on p.~\pageref{page:table}. Then the Lie algebra structure of $V_1^{\orb(\sigma_h)}$ is as follows:
\begin{equation*}
\renewcommand{\arraystretch}{1.2}
\begin{tabular}{rlrlcl}
No.\ & $V_1$ & $n$ & $V^{\sigma_h}_1$ & $d$ & $V^{\orb(\sigma_h)}_1$\\\hline
 $(1)$& $A_5C_5E_{6,2}$     &$2$& $A_5C_5D_5\C$         & $264$& $A_9^2D_6$     \\
 $(2)$& $A_3A_{7,2}C_3^2$   &$2$& $A_1^2A_3A_7B_2^2$    & $216$& $A_7^2D_5^2$   \\
 $(3)$& $A_{8,2}F_{4,2}$    &$2$& $A_8B_4$              & $240$& $A_8^3$        \\
 $(4)$& $B_8E_{8,2}$        &$2$& $D_8E_8$              & $744$& $E_8^3$        \\
 $(5)$& $A_2^2A_{5,2}^2B_2$ &$2$& $A_1A_2^2A_3A_5B_2\C$ & $168$& $A_5^4D_4$     \\
 $(6)$& $C_8F_4^2$          &$2$& $C_4^2F_4^2$          & $312$& $E_6^4$        \\
 $(7)$& $A_{4,2}^2C_{4,2}$  &$2$& $A_4^2B_2^2$          & $144$& $A_4^6$        \\
 $(8)$& $A_{2,2}^4D_{4,4}$  &$2$& $A_1^4A_2^4$          &$\s96$& $A_2^{12}$     \\
 $(9)$& $B_5E_{7,2}F_4$     &$2$& $A_1B_5D_6F_4$        & $312$& $A_{11}D_7E_6$ \\
$(10)$& $B_4C_6^2$          &$2$& $B_2B_4C_4C_6$        & $312$& $A_{11}D_7E_6$ \\
$(11)$& $A_{4,5}^2$         &$5$& $A_4\C^4$             & $144$& $A_4^6$        \\
$(12)$& $A_4A_{9,2}B_3$     &$4$& $A_1A_4A_7B_3\C$      & $456$& $D_{10}E_7^2$  \\
$(13)$& $B_6C_{10}$         &$4$& $A_7B_2B_6\C$         & $456$& $D_{10}E_7^2$  \\
$(14)$& $A_1C_{5,3}G_{2,2}$ &$6$& $A_1A_2C_4\C$         & $312$& $E_6^4$        \\
$(15)$& $A_{1,2}A_{3,4}^3$  &$8$& $A_3\C^7$             & $168$& $D_4^6$
\end{tabular}
\renewcommand{\arraystretch}{1}
\end{equation*}
Moreover, $V^{\orb(\sigma_h)}$ is isomorphic as \voa{} to $V_L$ where $L$ is exactly the lattice in Proposition~\ref{prop:part1}.
\end{prop}
\begin{proof}
The dimensions of the fixed-point Lie subalgebras $V_1^{\sigma_h^k}$ are easily computable, and so we may use the dimension formulae (Corollary~\ref{cor:dimform}) to compute $d=\dim(V_1^{\orb(\sigma_h)})$. This yields the values for $d$ listed in the table above. We observe that in each case $\dim(V^{\orb(\sigma_h)}_1)=\dim((V_L)_1)$.

The Lie algebra $V^{\sigma_h}_1$ has to be a fixed-point Lie subalgebra of $V^{\orb(\sigma_h)}_1$ under an automorphism of order dividing $n$, namely under the inverse orbifold automorphism with respect to $\sigma_h$ restricted to $V^{\orb(\sigma_h)}_1$. Then, using Kac's theory we determine that the Lie algebra listed in the proposition for $V^{\orb(\sigma_h)}_1$ is the only Lie algebra of dimension $d$ in Schellekens' list of which $V^{\sigma_h}_1$ is a possible fixed-point Lie subalgebra under an automorphism of order dividing $n$. This determines $V^{\orb(\sigma_h)}_1$.

In all of the cases, the Lie algebra $V^{\orb(\sigma_h)}_1$ has rank 24. By Corollary~1.4 in \cite{DM04b} we conclude that $V^{\orb(\sigma_h)}$ is isomorphic to a lattice \voa{}. Indeed, we see that $V^{\orb(\sigma_h)}\cong V_L$ where $L$ is exactly the Niemeier lattice in the corresponding entry of the table in Proposition~\ref{prop:part1}.
\end{proof}

\begin{rem}
For the cases of order~$n$ a prime, i.e.\ $n=2,5$, there is a nice alternative argument by symmetry. Recall that $V_1=(V_L^{\orb(\hat\nu)})_1$ by definition of $V$ and $V^{\sigma_h}_1\cong(V_L^{\hat\nu})_1$ by the choice of $\sigma_h$. In particular the dimensions coincide. Corollary~\ref{cor:dimform} applied to the orbifold construction with $V$ and $\sigma_h\in\Aut(V)$ and to the one with $V_L$ and $\hat\nu$ yields
\begin{align*}
\dim(V_1)+\dim(V^{\orb(\sigma_h)}_1)&=24+(n+1)\dim(V^{\sigma_h}_1),\\
\dim((V_L)_1)+\dim((V_L^{\orb(\hat\nu)})_1)&=24+(n+1)\dim((V_L^{\hat\nu})_1)
\end{align*}
so that with the above observation we obtain
\begin{equation*}
\dim((V_L)_1)=\dim(V^{\orb(\sigma_h)}_1).
\end{equation*}
\end{rem}

\section{Part 3: Unique Conjugacy Class}\label{sec:part3}
In this section we establish the uniqueness up to conjugacy of automorphisms of lattice \voa{}s satisfying certain conditions.

\subsection{Lemmas}
We begin with some general lemmas. The following is a generalisation of Lemma~3.6 in \cite{LS16b}:
\begin{lem}\label{lem:prop3.6_2}
Let $\g$ be a Lie algebra with a direct-sum decomposition
\begin{equation*}
\g=\g_1\oplus\ldots\oplus\g_n
\end{equation*}
for some $n\in\Ns$, and assume that $g$ is an automorphism of $\g$ of order~$n$ such that
\begin{equation*}
g(\g_i)=\g_{i+1}\text{ for }i=1,\ldots,n-1\text{ and }g(\g_n)=\g_1
\end{equation*}
(so that in particular all the ideals $\g_i$, $i=1,\ldots,n$, are isomorphic). Let $f\in\Inn(\g)$ be an inner automorphism of $\g$, and assume that $(fg)^n=\id$. Then $fg$ is conjugate to $g$ under an inner automorphism of $\g$.
\end{lem}
\begin{proof}
We can write $f$ uniquely as $f=f_1f_2\ldots f_n$ where $f_i\in\Inn(\g_i)$ is an inner automorphism of $\g_i$ and the $f_i$ commute pairwise. Let us define $a^b:=b^{-1}ab$. Then $f_i^{g^j}=g^{-j}f_ig^j\in\Inn(\g_{i-j})$ where we view the subtraction $i-j$ modulo $n$.

Using $g^n=\id$ we rewrite the relation $(fg)^n=\id$ as
\begin{align*}
\id&=(fg)^n=(f_1f_2\ldots f_ng)^n\\
&=(f_1f_2\ldots f_n)(f_1^{g^{n-1}}f_2^{g^{n-1}}\ldots f_n^{g^{n-1}})\ldots(f_1^gf_2^g\ldots f_n^g),
\end{align*}
which we reorder using the commutativity to obtain
\begin{equation*}
\id=\underbrace{(f_1f_n^{g^{n-1}}f_{n-1}^{g^{n-2}}\ldots f_2^g)}_{\in\Inn(\g_1)}\underbrace{(f_2f_1^{g^{n-1}}f_n^{g^{n-2}}\ldots f_3^g)}_{\in\Inn(\g_2)}\ldots\underbrace{(f_nf_{n-1}^{g^{n-1}}\ldots f_1^g)}_{\in\Inn(\g_n)}.
\end{equation*}
This implies that
\begin{align*}
f_1f_n^{g^{n-1}}f_{n-1}^{g^{n-2}}\ldots f_2^g&=\id,\\
f_2f_1^{g^{n-1}}f_n^{g^{n-2}}\ldots f_3^g&=\id,\\
&\;\;\vdots\\
f_nf_{n-1}^{g^{n-1}}\ldots f_1^g&=\id.
\end{align*}
We set
\begin{equation*}
x:=(f_1)(f_2f_1^{g^{n-1}})(f_3f_2^{g^{n-1}}f_1^{g^{n-2}})\ldots(f_{n-1}f_{n-2}^{g^{n-1}}\ldots f_1^{g^2})
\end{equation*}
where the $i$-th bracket is in $\Inn(\g_i)$ and hence all brackets commute. Using the above relations the inverse of $x$ is given by
\begin{equation*}
x^{-1}=(f_n^{g^{n-1}}f_{n-1}^{g^{n-2}}\ldots f_2^g)\ldots(f_n^{g^2}f_{n-1}^g)(f_n^g).
\end{equation*}
A simple calculation now shows that $xgx^{-1}=fg$.
\end{proof}

Recall that $\langle\cdot,\cdot\rangle$ also denotes the symmetric bilinear form on a lattice $L$ or on its complexification $\h=L\otimes_\Z\C$ and that we identify $\h$ with $\{h(-1)\otimes\ee_0\,|\,h\in\h\}\subseteq(V_L)_1$.
\begin{lem}[cf.\ Lemma~3.9 in \cite{LS16b}]\label{lem:trivialonla}
Let $L$ be a positive-definite, even lattice and $R$ the sublattice of $L$ generated by the vectors $\alpha\in L$ of norm $\langle\alpha,\alpha\rangle/2=1$, which we suppose to be of full rank. Let $g$ be an automorphism of the lattice \voa{} $V_L$ whose restriction to the reductive Lie algebra $(V_L)_1$ is trivial. Then $g=\e^{(2\pi\i)h_0}$ for some $h\in R'$, the dual lattice of $R$.
\end{lem}
\begin{proof}
Because $g$ acts trivially on the Cartan subalgebra $\hh=\{h(-1)\otimes\ee_0\,|\,h\in\h\}$ of $(V_L)_1$, by Lemma~2.5 in \cite{DN99} it is of the form $g=\e^{(2\pi\i)h_0}$ for some $h\in\h$. But $h$ acts as multiplication by $\langle h,\alpha\rangle$ on the graded piece $V_L(\alpha)\subseteq V_L$, and, by assumption, $\e^{(2\pi\i)\langle h,\alpha\rangle}=1$ for all $\alpha\in L$ of norm $\langle\alpha,\alpha\rangle/2=1$. This means that $\langle h,\alpha\rangle\in\Z$ for all $\alpha\in L$ of norm $\langle\alpha,\alpha\rangle/2=1$ and hence by linearity for all $\alpha\in R$. This proves the statement.
\end{proof}

The following lemma generalises Lemma~3.7 in \cite{LS16b}.
\begin{lem}\label{lem:conjproj}
Let $V$ be a \voa{} of CFT-type and assume that the Lie algebra $V_1$ is reductive. Let $\hh$ be a Cartan subalgebra of $V_1$ and $g$ an automorphism of $V$ of order $n\in\Ns$ such that $g(\hh)=\hh$. For $u\in\hh\subseteq V_1$ consider the inner automorphism $\e^{u_0}\in\Aut(V)$. Then $\e^{u_0}g$ is conjugate in $\Aut(V)$ to $\exp((\frac{1}{n}\sum_{i=0}^{n-1}g^iu)_0)g$ under an inner automorphism of $V$. 
\end{lem}
We observe that $u\mapsto\frac{1}{n}\sum_{i=0}^{n-1}g^iu$ is a projection from $\hh$ to the subspace of fixed points under $g$.
\begin{proof}
Note that all elements of $\hh$ commute with one another so that $[u_0,v_0]=0$ for $u,v\in\hh$ by Borcherds' identity. Moreover, $g\e^{u_0}=\e^{(gu)_0}g$ and $(\e^{u_0})^{-1}=\e^{(-u)_0}$ for any $u\in\hh$. We define $x:=\exp((\frac{1}{n}\sum_{i=1}^{n-1}ig^iu)_0)$ and calculate
\begin{align*}
x\e^{u_0}gx^{-1}&=\exp((\frac{1}{n}\sum_{i=1}^{n-1}ig^iu)_0)\exp(u_0)g\exp((-\frac{1}{n}\sum_{i=1}^{n-1}ig^iu)_0)\\
&=\exp((\frac{1}{n}\sum_{i=1}^{n}ig^iu)_0)\exp((-\frac{1}{n}\sum_{i=1}^{n-1}ig^{i+1}u)_0)g=\exp((\frac{1}{n}\sum_{i=0}^{n-1}g^iu)_0)g.
\end{align*}
This proves the statement.
\end{proof}
In the case of a lattice \voa{} $V_L$ associated with a positive-definite, even lattice $L$ choose the Cartan subalgebra $\hh=\{h(-1)\otimes\ee_0\,|\,h\in\h\}\cong\h$ and let $\hat\nu\in\Aut(V_L)$ be an automorphism obtained as lift of a lattice automorphism $\nu\in\Aut(L)$. The above lemma states that $\e^{h_0}\hat\nu$ is conjugate to $\e^{(\pi_\nu h)_0}\hat\nu$ where $\pi_\nu$ is the orthogonal projection from $\h$ to the complexified fixed-point sublattice $L^\nu\otimes_\Z\C$.

\begin{lem}\label{lem:innerouter}
Let $\g$ be a simple Lie algebra not of type $D_4$. Then any two outer automorphisms of $\g$ conjugate in $\Aut(\g)$ are in fact conjugate under an element of $\Inn(\g)$ (and also under an element of $\Aut(\g)\setminus\Inn(\g)$ if the latter is non-empty).
\end{lem}
\begin{proof}
If $\g$ is not of type $D_4$, then $\Out(\g)=\Aut(\g)/\Inn(\g)$ is either trivial or isomorphic to $\Z_2$. In the trivial case there is nothing to show, so assume the latter. In general, let $G$ be a group and $H\subseteq G$ a subgroup of index $2$. Let $a,b\in G\setminus H$ be conjugate under $x\in G$, i.e.\ $a=xbx^{-1}$. Then $a=(xb)b(xb)^{-1}$ too, and one of $x$ and $xb$ lies in $H$ while the other lies in $G\setminus H$.
\end{proof}

\subsection{Uniqueness}
In the following we use the above lemmas to prove the main result of this section:
\begin{prop}[Part 3]\label{prop:part3}
Let $L=N(Q)$ be a Niemeier lattice with root lattice~$Q$, $V_L$ the associated lattice \voa{} and $g$ an automorphism of $V_L$. Then for each entry of the following table the order~$n$ of $g$, the weight-one fixed-point Lie subalgebra $(V_L^g)_1$ and, if indicated, the value of $\rho(V_L(g))$ determine a unique conjugacy class in $\Aut(V_L)$:
\begin{equation*}
\renewcommand{\arraystretch}{1.2}
\begin{tabular}{rlrlc}
No.\  & $L$               & $n$ & $(V_L^g)_1$           & $\rho(V_L(g))$ \\\hline
 $(1)$& $N(A_9^2D_6)$     & $2$ & $A_5C_5D_5\C$         & $1$ \\
 $(2)$& $N(A_7^2D_5^2)$   & $2$ & $A_1^2A_3A_7B_2^2$    &     \\
 $(3)$& $N(A_8^3)$        & $2$ & $A_8B_4$              &     \\
 $(4)$& $N(E_8^3)$        & $2$ & $D_8E_8$              &     \\
 $(5)$& $N(A_5^4D_4)$     & $2$ & $A_1A_2^2A_3A_5B_2\C$ &     \\
 $(6)$& $N(E_6^4)$        & $2$ & $C_4^2F_4^2$          &     \\
 $(7)$& $N(A_4^6)$        & $2$ & $A_4^2B_2^2$          &     \\
 $(8)$& $N(A_2^{12})$     & $2$ & $A_1^4A_2^4$          &     \\
 $(9)$& $N(A_{11}D_7E_6)$ & $2$ & $A_1B_5D_6F_4$        &     \\
$(10)$& $N(A_{11}D_7E_6)$ & $2$ & $B_2B_4C_4C_6$        & $1$ \\
$(11)$& $N(A_4^6)$        & $5$ & $A_4\C^4$             & $1$ \\
$(12)$& $N(D_{10}E_7^2)$  & $4$ & $A_1A_4A_7B_3\C$      &     \\
$(13)$& $N(D_{10}E_7^2)$  & $4$ & $A_7B_2B_6\C$         &     \\
$(14)$& $N(E_6^4)$        & $6$ & $A_1A_2C_4\C$         &     \\
$(15)$& $N(D_4^6)$        & $8$ & $A_3\C^7$             & $1$
\end{tabular}
\renewcommand{\arraystretch}{1}
\end{equation*}
In particular, the automorphism $\hat\nu\in\Aut(V_L)$ from Proposition~\ref{prop:part1} and the inverse orbifold automorphism $\amgis\in\Aut(V^{\orb(\sigma_h)})$ with respect to the inner automorphism $\sigma_h$ (see table on p.~\pageref{page:table}) are conjugate under the isomorphism $V^{\orb(\sigma_h)}\cong V_L$.
\end{prop}

The proof of the proposition will be carried out in several steps. The first step is to use Kac's theory on fixed-point Lie subalgebras of simple Lie algebras (see Section~\ref{sec:kac}) to establish an analogous result at the level of conjugacy in $\Aut((V_L)_1)$.
\begin{lem}\label{lem:uniqueautlie}
For each entry of the table in Proposition~\ref{prop:part3} there is a unique conjugacy class in $\Aut((V_L)_1)$ of order~$n$ with fixed-point Lie subalgebra $(V_L^g)_1$.
\end{lem}
\begin{proof}
We apply the straightforward extension of Kac's theory to semisimple Lie algebras (see Section~\ref{sec:kac}) case by case.
\end{proof}

In the following we investigate whether two automorphisms of $(V_L)_1$ of order~$n$ with the same fixed-point Lie subalgebra are not just conjugate in $\Aut((V_L)_1)$ but more specifically conjugate under an element whose projection to $\Out((V_L)_1)=\Aut((V_L)_1)/\Inn((V_L)_1)$ lies in a certain subgroup $H\subseteq\Out((V_L)_1)$ that we shall describe below (see Lemma~\ref{lem:uniqueautlie2}). This is in particular fulfilled if the two automorphisms are conjugate under an inner automorphism, which corresponds to the trivial element in $\Out((V_L)_1)$.

The following summary is based on Sections 4.3, 16.1 and 18.4 of \cite{CS99} and partially draws from \cite{ISS15}, Section~2 and \cite{LS16b}, Section~3.3. Let $L=N(Q)$ be a Niemeier lattice other than the Leech lattice with root lattice $Q$ and simple roots $\Delta$. We denote by $H$ the subgroup of $\Aut(L)$ consisting of all elements that preserve $\Delta$ as a set. The elements of $H$ can be viewed as automorphisms of the Dynkin diagram of $Q$. This correspondence is injective but in general there are Dynkin diagram automorphisms that do not arise in this way. Since the elements of $\Aut(L)$ preserve $Q$, there is a homomorphism from $H$ into the automorphism group of the finite quadratic space $L/Q$, the glue code of $L$. In general this map is not injective. 

Moreover, one can show that $H\cong\Aut(L)/W$ where $W=G_0$ is the Weyl group of $L$ (or that of $Q$), i.e.\ the subgroup of $\Aut(L)$ generated by the reflections in the roots of $L$. $\Aut(L)$ is a split group extension $W{:}H$.

Let $G_1$ be the normal subgroup of $H$ consisting of those automorphisms that fix each irreducible component of the Dynkin diagram of $Q$. Then $H$ is a group extension $G_1.G_2$ where $G_2=H/G_1$ is isomorphic to a subgroup of the permutations of the isomorphic components of the Dynkin diagram of $Q$.

The Weyl group $G_0=W$ and the groups $G_1$ and $G_2$ are tabulated in \cite{CS99} for all Niemeier lattices.

Now let $V_L$ be the lattice \voa{} associated with $L=N(Q)$. Recall that $K=\langle\{\e^{v_0}\,|\,v\in(V_L)_1\}\rangle$ is the group of inner automorphisms of $V_L$. There is a surjective homomorphism
\begin{equation*}
p\colon\Aut(V_L)\to H
\end{equation*}
with kernel $K$, i.e.\ $\Aut(V_L)/K\cong H$ (see \cite{LS16b}, Lemma~3.10), and $H$ can be identified with a subgroup of the diagram automorphisms of the Dynkin diagram of $Q$, i.e.\ the outer automorphism group of $(V_L)_1$. Let
\begin{equation*}
r\colon\Aut(V_L)\to\Aut((V_L)_1)
\end{equation*}
denote the restriction from $V_L$ to $(V_L)_1$. Moreover, consider for any Lie algebra $\g$ the projection
\begin{equation*}
\pi\colon\Aut(\g)\to\Out(\g)=\Aut(\g)/\Inn(\g).
\end{equation*}
Finally, we denote by $q$ the composition
\begin{equation*}
q\colon\Aut(V_L)\to H\to G_2.
\end{equation*}

For an automorphism $g$ of $V_L$, $p(g)\in H$ can be identified with the element $\pi(r(g))\in\Out((V_L)_1)$, i.e.\ there is the following commutative diagram:
\begin{equation}\label{eq:cd}
\begin{tikzcd}
\Aut(V_L)\arrow{rr}{p}\arrow{dd}{r}&&H\arrow[hookrightarrow]{dd}\\
\\
\Aut((V_L)_1)\arrow{rr}{\pi}&&\Out((V_L)_1)
\end{tikzcd}
\end{equation}
The restriction $r\colon\Aut(V_L)\to\Aut((V_L)_1)$ is in general not surjective. Since $p$ is surjective, $r(\Aut(V_L))=\pi^{-1}(H)$, i.e.\ the group $\pi^{-1}(H)$ can be identified with the subgroup of $\Aut((V_L)_1)$ of automorphisms that lift to automorphisms of $V_L$.

\begin{lem}\label{lem:uniqueautlatvoa}
Under the assumptions of Proposition~\ref{prop:part3} the projection $p(g)$ is in a unique conjugacy class in $H$.
\end{lem}
\begin{proof}
We make a case-by-case analysis. In all of the cases, since the rank of $(V_L^g)_1$ is less than the rank of $(V_L)_1$, the automorphism $r(g)$ of the Lie algebra $(V_L)_1$ is outer, so $p(g)$ is non-trivial. If $n$ is prime, i.e.\ $n=2,5$, we deduce that $p(g)$ has order $m=n$. For the two cases in which $g$ has order $n=4$ we use Kac's theory to determine that $p(g)$ has order $m=2$, in the order~6 case $m=n=6$, and in the order $n=8$ case $p(g)$ has order $m=4$.

In each case we determine the group structure of $H$. If $H$ possesses a unique conjugacy class of elements of order~$m$, then we are done. If not, we use Kac's theory to determine all conjugacy classes of Lie algebra automorphisms of $(V_L)_1$ with fixed-point Lie subalgebra $(V_L^g)_1$. There can be more than one such conjugacy class in $\Aut((V_L)_1)$, but in all of the following cases they reduce to a single conjugacy class in $H$.

Let $\tau$ denote the diagram automorphism of order~2 of $A_l$ for $l\geq 2$, $D_l$ for $l\geq 5$ and $E_6$. For a generic simple type $X_l$ and $\{p_1,\ldots,p_s\}=\{1,2,\ldots,s\}$ we denote by $\sigma_{(p_1\,p_2\,\ldots\,p_s)}$ the permutation of the simple components of $X_l^s$ of cycle structure $(p_1\,p_2\,\ldots\,p_s)$. In the following discussion we list representatives of conjugacy classes of $H$ in terms of $\tau$ and the permutations of the simple components. In general, this will depend on the choice of the glue code for the Niemeier lattice (e.g. in case~(8) where the glue code is the ternary Golay code). However, this is not essential for identifying the conjugacy class.
\begin{enumerate}
\item For $L=N(A_9^2D_6)$ the group $H\cong\Z_4$ has exactly one element of order~2.

\item For $L=N(A_7^2D_5^2)$ the group $H\cong D_4$ (the dihedral group of order~8) has exactly three conjugacy classes of order~2. These can be distinguished by their projection to $G_2$:
\begin{equation*}
\renewcommand{\arraystretch}{1.2}
\begin{tabular}{cc|cc}
\multicolumn{2}{c|}{conj.\ class in $H$}&\multicolumn{2}{c}{proj.\ to $G_2$}\\
order & length & on $A_7^2$ & on $D_5^2$\\\hline
$2$ & $1$ & $\id$ & $\id$\\
$2$ & $2$ & $\id$ & $\sigma_{(1\,2)}$\\
$2$ & $2$ & $\sigma_{(1\,2)}$ & $\id$
\end{tabular}
\renewcommand{\arraystretch}{1}
\end{equation*}

By Kac's theory, the automorphism $r(g)$ of $(V_L)_1\cong A_7^2D_5^2$ of order~2 with fixed-point Lie subalgebra $A_1^2A_3A_7B_2^2$ can only occur as
\begin{equation*}
A_7^2\leadsto A_7,\quad D_5\leadsto A_1^2A_3,\quad D_5\leadsto B_2^2
\end{equation*}
up to reordering of the two copies of $D_5$ with
\begin{equation*}
q(g)=\begin{cases}\sigma_{(1\,2)}&\text{on }A_7^2,\\\id&\text{on }D_5^2\end{cases}
\end{equation*}
so that $p(g)$ has to be in the third conjugacy class.

\item For $L=N(A_8^3)$ the group $H\cong D_6$ (the dihedral group of order~12) has exactly three conjugacy classes of order~2. In this case it is not enough to study the projection to $G_2$. The conjugacy classes can be described as:
\begin{equation*}
\renewcommand{\arraystretch}{1.2}
\begin{tabular}{ccc}
\multicolumn{3}{c}{conj.\ class in $H$}\\
order & length & on $A_8^3$ \\\hline
$2$ & $1$ & $\tau_{(1)}\tau_{(2)}\tau_{(3)}$\\
$2$ & $3$ & $\sigma_{(1\,2)},\sigma_{(2\,3)},\sigma_{(1\,3)}$\\
$2$ & $3$ & $\sigma_{(1\,2)}\tau_{(3)},\sigma_{(2\,3)}\tau_{(1)},\sigma_{(1\,3)}\tau_{(2)}$
\end{tabular}
\renewcommand{\arraystretch}{1}
\end{equation*}
where $\tau_{(i)}$ denotes the diagram automorphism $\tau$ (see above) on the $i$-th copy $A_8^{(i)}$ of $A_8$, $i=1,2,3$.

By Kac's theory, the automorphism $r(g)$ of $(V_L)_1\cong A_8^3$ of order~2 with fixed-point Lie subalgebra $A_8B_4$ can only occur as
\begin{equation*}
A_8^2\leadsto A_8,\quad A_8\leadsto B_4
\end{equation*}
up to reordering of the three copies of $A_8$ with
\begin{equation*}
p(g)=\begin{cases}\sigma_{(1\,2)}&\text{on }A_8^{(1)}A_8^{(2)},\\\tau_{(3)}&\text{on }A_8^{(3)},\end{cases}
\end{equation*}
for example, so that $p(g)$ has to be in the third conjugacy class.

\item For $L=N(E_8^3)$ the group $H\cong S_3$ has exactly one conjugacy class of order~2, which has length 3.

\item For $L=N(A_5^4D_4)$ the group $H\cong\GL(2,3)$ has exactly two conjugacy classes of order~2. These can be distinguished by their projection to $G_2$:
\begin{equation*}
\renewcommand{\arraystretch}{1.2}
\begin{tabular}{cc|cc}
\multicolumn{2}{c|}{conj.\ class in $H$}&\multicolumn{2}{c}{proj.\ to $G_2$}\\
order & length & on $A_5^4$ & on $D_4$\\\hline
$2$ &$\s1$ & $\id$ & $\id$\\
$2$ & $12$ & $\sigma_{(1\,2)},\ldots$ & $\id$ \\
\end{tabular}
\renewcommand{\arraystretch}{1}
\end{equation*}

By Kac's theory, the automorphism $r(g)$ of $(V_L)_1\cong A_5^4D_4$ of order~2 with fixed-point Lie subalgebra $A_1A_2^2A_3A_5B_2\C$ can only occur as
\begin{equation*}
A_5^2\leadsto A_5,\quad A_5\leadsto A_2^2\C,\quad A_5\leadsto A_3,\quad D_4\leadsto A_1B_2
\end{equation*}
up to reordering of the four copies of $A_5$ with
\begin{equation*}
q(g)=\begin{cases}\sigma_{(1\,2)}&\text{on }A_5^4,\\\id&\text{on }D_4,\end{cases}
\end{equation*}
for example, so that $p(g)$ has to be in the second conjugacy class.

\item For $L=N(E_6^4)$ the group $H\cong\GL(2,3)$ has exactly two conjugacy classes of order~2. These can be distinguished by their projection to $G_2$:
\begin{equation*}
\renewcommand{\arraystretch}{1.2}
\begin{tabular}{cc|c}
\multicolumn{2}{c|}{conj.\ class in $H$}&proj.\ to $G_2$\\
order & length & on $E_6^4$ \\\hline
$2$ &$\s1$ & $\id$ \\
$2$ & $12$ & $\sigma_{(1\,2)},\ldots$ \\
\end{tabular}
\renewcommand{\arraystretch}{1}
\end{equation*}

By Kac's theory, the automorphism $r(g)$ of $(V_L)_1\cong E_6^4$ of order~2 with fixed-point Lie subalgebra $C_4^2F_4^2$ can only occur as
\begin{equation*}
E_6\leadsto C_4,\quad E_6\leadsto C_4,\quad E_6\leadsto F_4,\quad E_6\leadsto F_4
\end{equation*}
up to reordering of the four copies of $E_6$ with
\begin{equation*}
q(g)=\id
\end{equation*}
so that $p(g)$ has to be in the first conjugacy class.

\item For $L=N(A_4^6)$ the group $H$ (a group extension $\Z_2.S_5$) has exactly three conjugacy classes of order~2. In this case it is not enough to study the projection to $G_2$. The conjugacy classes can be described as:
\begin{equation*}
\renewcommand{\arraystretch}{1.2}
\begin{tabular}{ccc}
\multicolumn{3}{c}{conj.\ class in $H$}\\
order & length & on $A_4^6$ \\\hline
$2$ &$\s1$ & $\tau_{(1)}\tau_{(2)}\tau_{(3)}\tau_{(4)}\tau_{(5)}\tau_{(6)}$\\
$2$ & $15$ & $\sigma_{(1\,2)}\sigma_{(3\,4)},\ldots$ \\
$2$ & $15$ & $\sigma_{(1\,2)}\sigma_{(3\,4)}\tau_{(5)}\tau_{(6)},\ldots$
\end{tabular}
\renewcommand{\arraystretch}{1}
\end{equation*}
where $\tau_{(i)}$ denotes the diagram automorphism $\tau$ on the $i$-th copy $A_4^{(i)}$ of $A_4$, $i=1,\ldots,6$.

By Kac's theory, the automorphism $r(g)$ of $(V_L)_1\cong A_4^6$ of order~2 with fixed-point Lie subalgebra $A_4^2B_2^2$ can only occur as
\begin{equation*}
A_4^2\leadsto A_4,\quad A_4^2\leadsto A_4,\quad A_4\leadsto B_2,\quad A_4\leadsto B_2
\end{equation*}
up to reordering of the six copies of $A_4$ with
\begin{equation*}
p(g)=\begin{cases}\sigma_{(1\,2)}&\text{on }A_4^{(1)}A_4^{(2)},\\\sigma_{(3\,4)}&\text{on }A_4^{(3)}A_4^{(4)},\\\tau_{(5)}&\text{on }A_4^{(5)},\\\tau_{(6)}&\text{on }A_4^{(6)},\end{cases}
\end{equation*}
for example, so that $p(g)$ has to be in the second conjugacy class.

\item For $L=N(A_2^{12})$ the group $H$ (a group extension $\Z_2.M_{12}$) has exactly three conjugacy classes of order~2. Also in this case it is not enough to study the projection to $G_2$. The conjugacy classes can be described as:
\begin{equation*}
\renewcommand{\arraystretch}{1.2}
\begin{tabular}{ccc}
\multicolumn{3}{c}{conj.\ class in $H$}\\
order & length & on $A_2^{12}$ \\\hline
$2$ &$\s\s1$ & $\tau_{(1)}\tau_{(2)}\ldots\tau_{(12)}$\\
$2$ &  $495$ & $\sigma_{(1\,2)}\sigma_{(3\,4)}\sigma_{(5\,6)}\sigma_{(7\,8)},\ldots$\\
$2$ &  $495$ & $\sigma_{(1\,2)}\sigma_{(3\,4)}\sigma_{(5\,6)}\sigma_{(7\,8)}\tau_{(9)}\tau_{(10)}\tau_{(11)}\tau_{(12)},\ldots$
\end{tabular}
\renewcommand{\arraystretch}{1}
\end{equation*}
where $\tau_{(i)}$ denotes the diagram automorphism $\tau$ on the $i$-th copy $A_2^{(i)}$ of $A_2$, $i=1,\ldots,12$.

By Kac's theory, the automorphism $r(g)$ of $(V_L)_1\cong A_2^{12}$ of order~2 with fixed-point Lie subalgebra $A_1^4A_2^4$ can only occur as
\begin{equation*}
A_2^2\leadsto A_2\text{ (4 times)},\quad A_2\leadsto A_1\text{ (4 times)}
\end{equation*}
up to reordering of the twelve copies of $A_2$ with
\begin{equation*}
p(g)=\begin{cases}\sigma_{(1\,2)}&\text{on }A_2^{(1)}A_2^{(2)},\\\sigma_{(3\,4)}&\text{on }A_2^{(3)}A_2^{(4)},\\\sigma_{(5\,6)}&\text{on }A_2^{(5)}A_2^{(6)},\\\sigma_{(7\,8)}&\text{on }A_2^{(7)}A_2^{(8)},\\\tau_{(9)}&\text{on }A_2^{(9)},\\\tau_{(10)}&\text{on }A_2^{(10)},\\\tau_{(11)}&\text{on }A_2^{(11)},\\\tau_{(12)}&\text{on }A_2^{(12)},\end{cases}
\end{equation*}
for example, so that $p(g)$ has to be in the third conjugacy class.

\item For $L=N(A_{11}D_7E_6)$ the group $H\cong\Z_2$ has exactly one element of order~2.

\item As in the previous case.

\item For $L=N(A_4^6)$ the group $H$ (a group extension $\Z_2.S_5$) has exactly one conjugacy class of order~5, which has length 24.

\item For $L=N(D_{10}E_7^2)$ the group $H\cong\Z_2$ has exactly one element of order~2.

\item As in the previous case.

\item For $L=N(E_6^4)$ the group $H\cong\GL(2,3)$ has exactly one conjugacy class of order~6, which has length 8.

\item For $L=N(D_4^6)$ the group $H$ (a group extension $\Z_3.S_6$) has exactly two conjugacy classes of order~4. These can be distinguished by their projection to $G_2$:
\begin{equation*}
\renewcommand{\arraystretch}{1.2}
\begin{tabular}{cc|c}
\multicolumn{2}{c|}{conj.\ class in $H$}&proj.\ to $G_2$\\
order & length & on $D_4^6$ \\\hline
$4$ &$\s90$ & $\sigma_{(1\,2\,3\,4)}\sigma_{(5\,6)},\ldots$ \\
$4$ & $270$ & $\sigma_{(1\,2\,3\,4)},\ldots$ \\
\end{tabular}
\renewcommand{\arraystretch}{1}
\end{equation*}

By Kac's theory, the automorphism $r(g)$ of $(V_L)_1\cong D_4^6$ of order~8 with fixed-point Lie subalgebra $A_3\C^7$ can only occur as
\begin{equation*}
D_4^4\leadsto A_3\C,\quad D_4\leadsto\C^3,\quad D_4\leadsto\C^3
\end{equation*}
up to reordering of the six copies of $D_4$ with
\begin{equation*}
q(g)=\begin{cases}\sigma_{(1\,2\,3\,4)}&\text{on }D_4^{(1)}\ldots D_4^{(4)},\\\id&\text{on }D_4^{(5)},\\\id&\text{on }D_4^{(6)},\end{cases}
\end{equation*}
for example, so that $p(g)$ has to be in the second conjugacy class.\qedhere
\end{enumerate}
\end{proof}

Instead of just studying conjugacy classes in $\Aut((V_L)_1)$ as in Lemma~\ref{lem:uniqueautlie} we now study conjugacy in $\Aut((V_L)_1)$ under automorphisms in $\pi^{-1}(H)$, i.e.\ under those automorphisms of $(V_L)_1$ which lift to automorphisms of $V_L$ (see the diagram~\eqref{eq:cd}), e.g.\ under inner automorphisms.
\begin{lem}\label{lem:uniqueautlie2}
Under the assumptions of Proposition~\ref{prop:part3} the restriction $r(g)\in\Aut((V_L)_1)$ lies in a unique class in $\Aut((V_L)_1)$ up to conjugacy under automorphisms in $\pi^{-1}(H)$.
\end{lem}
\begin{proof}
Let $h$ and $g$ be two automorphisms of $V_L$ of order~$n$ and with weight-one fixed-points $(V_L^h)_1\cong (V_L^g)_1$ as in the assumptions of Proposition~\ref{prop:part3}. We aim to show that $r(h)$ and $r(g)$ in $\Aut((V_L)_1)$ are conjugate under an automorphism in $\pi^{-1}(H)$. To this end we make a case-by-case analysis using Lemmas \ref{lem:prop3.6_2}, \ref{lem:innerouter}, \ref{lem:uniqueautlie} and \ref{lem:uniqueautlatvoa}.

Note that $p(g)$ and $p(h)$ are by Lemma~\ref{lem:uniqueautlatvoa} conjugate in $H$, i.e.\ $h$ is conjugate to $xg$ in $\Aut(V_L)$ for some inner automorphism $x\in K=\langle\{\e^{v_0}\,|\,v\in (V_L)_1\}\rangle$. Restricting to $(V_L)_1$ we can state that $r(h)$ is conjugate to $xr(g)$ for some $x\in\Inn((V_L)_1)$, conjugate under an automorphism in $r(\Aut(V_L))=\pi^{-1}(H)$. Hence, we may assume in the proof that $r(h)=xr(g)$ for some $x\in\Inn((V_L)_1)$.
\begin{enumerate}
\item[(9)] By Kac's theory, the automorphisms $r(h)$ and $r(g)$ of $(V_L)_1\cong A_{11}D_7E_6$ of order~2 with fixed-point Lie subalgebra $A_1B_5D_6F_4$ can only occur as
\begin{equation*}
A_{11}\leadsto D_6,\quad D_7\leadsto A_1B_5,\quad E_6\leadsto C_4.
\end{equation*}

We can study the restriction to each simple component separately. In all three cases there is a unique conjugacy class of automorphisms of order~2 of the simple component with the given fixed-point Lie subalgebra. Since all of the automorphisms must be outer, by Lemma~\ref{lem:innerouter} there is even a unique class up to conjugacy under inner automorphisms.

Hence $r(h)$ and $r(g)$ are conjugate under an inner automorphism of $A_{11}D_7E_6$.

\item[(10)] By Kac's theory, the automorphisms $r(h)$ and $r(g)$ of $(V_L)_1\cong A_{11}D_7E_6$ of order~2 with fixed-point Lie subalgebra $B_2B_4C_4C_6$ can only occur as
\begin{equation*}
A_{11}\leadsto C_6,\quad D_7\leadsto B_2B_4,\quad E_6\leadsto C_4.
\end{equation*}

With exactly the same reasoning as in case (9) we conclude that $r(h)$ and $r(g)$ are conjugate under an inner automorphism of $A_{11}D_7E_6$.

\item[(6)] By Kac's theory, the automorphisms $r(h)$ and $r(g)$ of $(V_L)_1\cong E_6^4$ of order~2 with fixed-point Lie subalgebra $C_4^2F_4^2$ can only occur as
\begin{equation*}
E_6\leadsto C_4,\quad E_6\leadsto C_4,\quad E_6\leadsto F_4,\quad E_6\leadsto F_4
\end{equation*}
up to reordering of the four copies of $E_6$.

Since the group $G_2\cong S_4$ is the full symmetric group permuting the four copies of $E_6$, it is possible to find elements $y_1,y_2\in H$
such that conjugating $r(h)$ and $r(g)$ by elements in $\pi^{-1}(\{y_1\})$ and $\pi^{-1}(\{y_2\})$, respectively, puts $r(h)$ and $r(g)$ into such a form that their fixed points coincide and occur exactly in the above specified order.

Now we can study the restriction of $r(h)$ and $g(h)$ to each simple component separately. With exactly the same reasoning as in case (9) we conclude that $r(h)$ and $r(g)$ are conjugate under an automorphism of $E_6^4$ in $\pi^{-1}(H)$.

\item[(3)] By Kac's theory, the automorphisms $r(h)$ and $r(g)$ of $(V_L)_1\cong A_8^3$ of order~2 with fixed-point Lie subalgebra $A_8B_4$ can only occur as
\begin{equation*}
A_8^2\leadsto A_8,\quad A_8\leadsto B_4
\end{equation*}
up to reordering of the three copies of $A_8$.

Recall that we may assume that $r(h)=xr(g)$ for some $x\in\Inn((V_L)_1)$. But then the fixed points of $r(h)$ and $r(g)$ must occur in the same order, and without loss of generality we may assume that they occur in the order specified above.

Then by Lemma~\ref{lem:prop3.6_2} the restrictions of $r(h)$ and $r(g)$ to the first two copies of $A_8$ are conjugate under an inner automorphism. On the third copy of $A_8$ the automorphisms $r(g)$ and $r(h)$ are outer and hence conjugate under an inner automorphism by Lemma~\ref{lem:innerouter}.

Hence $r(h)$ and $r(g)$ are conjugate under an inner automorphism of $A_8^3$.

\item[(4)] By Kac's theory, the automorphisms $r(h)$ and $r(g)$ of $(V_L)_1\cong E_8^3$ of order~2 with fixed-point Lie subalgebra $D_8E_8$ can only occur as
\begin{equation*}
E_8^2\leadsto E_8,\quad E_8\leadsto D_8
\end{equation*}
up to reordering of the three copies of $E_8$.

Once again we may take $r(h)=xr(g)$ for some $x\in\Inn((V_L)_1)$, and so without loss of generality the fixed points of $r(g)$ and $r(h)$ occur in the order specified above.

Then by Lemma~\ref{lem:prop3.6_2} the restrictions of $r(h)$ and $r(g)$ to the first two copies of $E_8$ are conjugate under an inner automorphism. On the third copy of $E_8$ the automorphisms $r(g)$ and $r(h)$ are conjugate and hence conjugate under an inner automorphism since $\Out(E_8)$ is trivial.

Hence $r(h)$ and $r(g)$ are conjugate under an inner automorphism of $E_8^3$.

\item[(7)] By Kac's theory, the automorphisms $r(h)$ and $r(g)$ of $(V_L)_1\cong A_4^6$ of order~2 with fixed-point Lie subalgebra $A_4^2B_2^2$ can only occur as
\begin{equation*}
A_4^2\leadsto A_4,\quad A_4^2\leadsto A_4,\quad A_4\leadsto B_2,\quad A_4\leadsto B_2
\end{equation*}
up to reordering of the six copies of $A_4$.

Once again we may take $r(h)=xr(g)$ for some $x\in\Inn((V_L)_1)$, and so without loss of generality the fixed points of $r(g)$ and $r(h)$ occur in the order specified above.

As before, by Lemma~\ref{lem:prop3.6_2}, the restrictions of $r(h)$ and $r(g)$ to the first two and second two copies of $A_4$ lie in the same class of automorphisms up to conjugation by inner automorphisms.

Since the restrictions of $r(h)$ and $r(g)$ to the second last and last copy of $A_4$ must be outer, by Lemma~\ref{lem:innerouter} they are conjugate under an inner automorphism.

Hence $r(h)$ and $r(g)$ are conjugate under an inner automorphism of $A_4^6$.

\item[(8)] By Kac's theory, the automorphisms $r(h)$ and $r(g)$ of $(V_L)_1\cong A_2^{12}$ of order~2 with fixed-point Lie subalgebra $A_1^4A_2^4$ can only occur as
\begin{equation*}
A_2^2\leadsto A_2\text{ (4 times)},\quad A_2\leadsto A_1\text{ (4 times)}
\end{equation*}
up to reordering of the twelve copies of $A_2$.

Once again we may take $r(h)=xr(g)$ for some $x\in\Inn((V_L)_1)$, and so without loss of generality the fixed points of $r(g)$ and $r(h)$ occur in the order specified above.

Exactly as in the previous case, using Lemmas \ref{lem:prop3.6_2} and \ref{lem:innerouter}, we conclude that $r(h)$ and $r(g)$ are conjugate under an inner automorphism of $A_2^{12}$.

\item[(11)] By Kac's theory, the automorphisms $r(h)$ and $r(g)$ of $(V_L)_1\cong A_4^6$ of order~5 with fixed-point Lie subalgebra $A_4\C^4$ can only occur as
\begin{equation*}
A_4^5\leadsto A_4,\quad A_4\leadsto\C^4
\end{equation*}
up to reordering of the six copies of $A_4$.

Once again we may take $r(h)=xr(g)$ for some $x\in\Inn((V_L)_1)$, and so without loss of generality the fixed points of $r(g)$ and $r(h)$ occur in the order specified above.

As before, by Lemma~\ref{lem:prop3.6_2}, the restrictions of $r(h)$ and $r(g)$ to the first five copies of $A_4$ lie in the same class of automorphisms up to conjugation by inner automorphisms.

An automorphism of $A_4$ with fixed-point Lie subalgebra $\C^4$ is necessarily inner. All automorphisms of order~5 of $A_4$ with fixed-point Lie subalgebra $\C^4$ are conjugate and even conjugate under an inner automorphism (see Section~\ref{sec:kac}).

Hence $r(h)$ and $r(g)$ are conjugate under an inner automorphism of $A_4^6$.

\item[(2)] By Kac's theory, the automorphisms $r(h)$ and $r(g)$ of $(V_L)_1\cong A_7^2D_5^2$ of order~2 with fixed-point Lie subalgebra $A_1^2A_3A_7B_2^2$ can only occur as
\begin{equation*}
A_7^2\leadsto A_7,\quad D_5\leadsto A_1^2A_3,\quad D_5\leadsto B_2^2
\end{equation*}
up to reordering of the two copies of $D_5$.

Once again we may take $r(h)=xr(g)$ for some $x\in\Inn((V_L)_1)$, and so without loss of generality the fixed points of $r(g)$ and $r(h)$ occur in the order specified above. Note that on the first copy of $D_5$ the automorphisms are inner, and on the second they are outer.

The restrictions of $r(h)$ and $r(g)$ to $A_7^2$ lie in the same class of automorphisms up to conjugation by inner automorphisms by Lemma~\ref{lem:prop3.6_2}.

The restrictions of $r(h)$ and $r(g)$ to the second copy of $D_5$ are outer and hence lie in the same class up to conjugacy under inner automorphisms by Lemma~\ref{lem:innerouter}.

Restricted to the first copy of $D_5$ the automorphisms $r(h)$ and $r(g)$ are inner, but all automorphisms of order~2 of $D_5$ with fixed-point Lie subalgebra $A_1^2A_3$ are conjugate and even conjugate under an inner automorphism (see Section~\ref{sec:kac}).

Hence $r(h)$ and $r(g)$ are conjugate under an inner automorphism of $A_7^2D_5^2$.

\item[(5)] By Kac's theory, the automorphisms $r(h)$ and $r(g)$ of $(V_L)_1\cong A_5^4D_4$ of order~2 with fixed-point Lie subalgebra $A_1A_2^2A_3A_5B_2\C$ can only occur as
\begin{equation*}
A_5^2\leadsto A_5,\quad A_5\leadsto A_2^2\C,\quad A_5\leadsto A_3,\quad D_4\leadsto A_1B_2
\end{equation*}
up to reordering of the four copies of $A_5$.

Once again we may take $r(h)=xr(g)$ for some $x\in\Inn((V_L)_1)$, and so without loss of generality the fixed points of $r(g)$ and $r(h)$ occur in the order specified above. Note that on the third copy of $A_5$ the automorphisms are inner, and on the fourth they are outer.

The restrictions of $r(h)$ and $r(g)$ to $A_5^2$ lie in the same class of automorphisms up to conjugation by inner automorphisms by Lemma~\ref{lem:prop3.6_2}.

Restricted to the third copy of $A_5$ the automorphisms $r(h)$ and $r(g)$ are inner, but all automorphisms of order~2 of $A_5$ with fixed-point Lie subalgebra $A_2^2\C$ are conjugate and even conjugate under an inner automorphism (see Section~\ref{sec:kac}).

The restrictions of $r(h)$ and $r(g)$ to the fourth copy of $A_5$ are outer and hence lie in the same class up to conjugacy under inner automorphisms by Lemma~\ref{lem:innerouter}.

The restrictions of $r(h)$ and $r(g)$ to $D_4$ are outer and square to inner automorphisms. By Kac's theory $r(h)$ and $r(g)$ are conjugate on $D_4$. Since we may assume that $r(h)=xr(g)$ for some $x\in\Inn((V_L)_1)$, both $r(h)$ and $r(g)$ project to the same transposition in $\Out(D_4)\cong S_3$. This implies that they can only be conjugate under an element that projects to the identity or to this transposition. In fact, they are conjugate under both.

Hence $r(h)$ and $r(g)$ are conjugate under an inner automorphism of $A_5^4D_4$.

\item[(1)] By Kac's theory, the automorphisms $r(h)$ and $r(g)$ of $(V_L)_1\cong A_9^2D_6$ of order~2 with fixed-point Lie subalgebra $A_5C_5D_5\C$ can only occur as
\begin{equation*}
A_9\leadsto C_5,\quad A_9\leadsto D_5,\quad D_6\leadsto A_5\C
\end{equation*}
up to reordering of the two copies of $A_9$.

Since the group $G_2\cong S_2$ is the full symmetric group permuting the two copies of $A_9$, it is possible to find elements $y_1,y_2\in H$
such that conjugating $r(h)$ and $r(g)$ by elements in $\pi^{-1}(\{y_1\})$ and $\pi^{-1}(\{y_2\})$, respectively, puts $r(h)$ and $r(g)$ into such a form that their fixed points coincide and occur exactly in the above specified order.

For each of the three simple components there is a unique conjugacy class of its automorphisms with the given fixed-point Lie subalgebra. On the two copies of $A_9$ the automorphisms must be outer, and by Lemma~\ref{lem:innerouter} there is a unique class up to conjugacy under inner automorphisms.

The automorphism on $D_6$ is inner. While there is a unique conjugacy class of automorphisms of $D_6$ of order~2 with fixed-point Lie subalgebra $A_5\C$, two such automorphisms need not be conjugate under an inner automorphism. Indeed, the conjugacy class in $\Aut(D_6)$ splits up into two classes with respect to conjugacy under $\Inn(D_6)$. If the restrictions of $r(h)$ and $r(g)$ to $D_6$ are conjugate under an inner automorphism, we are done. So let us assume they are conjugate under an outer automorphism, i.e.\ under an automorphism whose projection to $\Out(D_6)$ is $\tau$. The group $H\cong\Z_4$ contains two elements of order~4, namely $\sigma_{(1\,2)}\tau_{(1)}\tau_{(3)}$ and $\sigma_{(1\,2)}\tau_{(2)}\tau_{(3)}$. Let $y\in H$ be one of those two elements. Conjugating $r(h)$ by an element in $\pi^{-1}(\{y\})$ does not change the order of its fixed points on the two copies of $A_9$ and it brings it to a form such that the restrictions of $r(h)$ and $r(g)$ to $D_6$ are conjugate under an inner automorphism. So again we are done.

Hence $r(h)$ and $r(g)$ are conjugate under an automorphism of $A_9^2D_6$ in $\pi^{-1}(H)$.

\item[(12)] By Kac's theory, the automorphisms $r(h)$ and $r(g)$ of $(V_L)_1\cong D_{10}E_7^2$ of order~4 with fixed-point Lie subalgebra $A_1A_4A_7B_3\C$ can only occur as
\begin{equation*}
D_{10}\leadsto A_1A_4B_3\C,\quad E_7^2\leadsto A_7.
\end{equation*}

The restrictions of $r(h)$ and $r(g)$ to $E_7^2$ are conjugate under some element in $\Aut(E_7^2)$. First assume that this element is inner. Then, since the restrictions of $r(h)$ and $r(g)$ to $D_{10}$ are outer and hence conjugate under an inner automorphism by Lemma~\ref{lem:innerouter}, $r(h)$ and $r(g)$ are in total conjugate under an inner automorphism.

If on the other hand the restrictions of $r(h)$ and $r(g)$ to $E_7^2$ are conjugate under an element of $\Aut(E_7^2)$ that projects to $\sigma_{(2\,3)}\in\Out(E_7^2)\cong\Z_2$, then we note that, since the restrictions of $r(h)$ and $r(g)$ to $D_{10}$ are outer, they are by Lemma~\ref{lem:innerouter} also conjugate under an outer automorphism. So $r(h)$ and $r(g)$ are in total conjugate under an element in $\Aut(D_{10}E_7^2)$ that projects to $\tau_{(1)}\sigma_{(2\,3)} \in\Out(D_{10}E_7^2)$. Since $H=\langle\tau_{(1)}\sigma_{(2\,3)}\rangle\cong\Z_2$, this implies that $r(h)$ and $r(g)$ are conjugate under an automorphism of $D_{10}E_7^2$ in $\pi^{-1}(H)$.

\item[(13)] By Kac's theory, the automorphisms $r(h)$ and $r(g)$ of $(V_L)_1\cong D_{10}E_7^2$ of order~4 with fixed-point Lie subalgebra $A_7B_2B_6\C$ can only occur as
\begin{equation*}
D_{10}\leadsto B_2B_6\C,\quad E_7^2\leadsto A_7.
\end{equation*}

With exactly the same reasoning as in case (12) we conclude that $r(h)$ and $r(g)$ are conjugate under an automorphism of $D_{10}E_7^2$ in $\pi^{-1}(H)$.

\item[(14)] By Kac's theory, the automorphisms $r(h)$ and $r(g)$ of $(V_L)_1\cong E_6^4$ of order~6 with fixed-point Lie subalgebra $A_1A_2C_4\C$ can only occur as
\begin{equation*}
E_6^3\leadsto C_4,\quad E_6\leadsto A_1A_2\C
\end{equation*}
up to reordering of the four copies of $E_6$.

Once again we may take $r(h)=xr(g)$ for some $x\in\Inn((V_L)_1)$, and so without loss of generality the fixed points of $r(g)$ and $r(h)$ occur in the order specified above.

The restrictions of $r(h)$ and $r(g)$ to the fourth copy of $E_6$ are outer and hence lie in the same class up to conjugacy under inner automorphisms by Lemma~\ref{lem:innerouter}.

We slightly modify the argument of Lemma~\ref{lem:prop3.6_2} in order to show that the restrictions of $r(h)$ and $r(g)$ to the first three copies of $E_6$ are conjugate under an inner automorphism. We can write $r(g)$ without loss of generality as
\begin{equation*}
(v_1,v_2,v_3)\mapsto(\varphi_3v_3,\varphi_1v_1,\varphi_2v_2)
\end{equation*}
for $(v_1,v_2,v_3)\in E_6^3$ with $\varphi_1,\varphi_2,\varphi_3\in\Aut(E_6)$ and $\varphi_3\varphi_2\varphi_1$ outer, of order~2 and with fixed points $C_4$. Then, writing $x=(x_1,x_2,x_3)$ with $x_1,x_2,x_3\in\Inn(E_6)$, $r(h)=xr(g)$ is given by
\begin{equation*}
(v_1,v_2,v_3)\mapsto(x_1\varphi_3v_3,x_2\varphi_1v_1,x_3\varphi_2v_2)
\end{equation*}
for $(v_1,v_2,v_3)\in E_6^3$. Also $x_1\varphi_3x_3\varphi_2x_2\varphi_1$ is outer, of order~2 and has fixed points $C_4$. By Kac's theory $\varphi_3\varphi_2\varphi_1$ and $x_1\varphi_3x_3\varphi_2x_2\varphi_1$ are conjugate in $\Aut(E_6)$, and by Lemma~\ref{lem:innerouter} they are even conjugate under an inner automorphism, which we call $\xi\in\Inn(E_6)$, i.e.\ $\xi^{-1}x_1\varphi_3x_3\varphi_2x_2\varphi_1\xi=\varphi_3\varphi_2\varphi_1$. A simple calculation then shows that
\begin{equation*}
y^{-1}r(h)y=r(g)
\end{equation*}
with $y=(\xi,x_2\varphi_1\xi\varphi_1^{-1},x_3\varphi_2x_2\varphi_1\xi\varphi_1^{-1}\varphi_2^{-1})$. Since $\xi,x_2,x_3$ are inner, so is $y$, i.e.\ the projection of $y$ to $\Out(E_6^3)$ is trivial.

Hence $r(h)$ and $r(g)$ are conjugate under an inner automorphism of $E_6^4$.

\item[(15)] By Kac's theory, the automorphisms $r(h)$ and $r(g)$ of $(V_L)_1\cong D_4^6$ of order~8 with fixed-point Lie subalgebra $A_3\C^7$ can only occur as
\begin{equation*}
D_4^4\leadsto A_3\C,\quad D_4\leadsto\C^3,\quad D_4\leadsto\C^3
\end{equation*}
up to reordering of the six copies of $D_4$.

Since we may assume that $r(h)=xr(g)$ for some $x\in\Inn((V_L)_1)$, in this case the fixed-points of $r(h)$ and $r(g)$ occur in the same order, and we may as well assume that they occur in the above specified order.

Then by Lemma~\ref{lem:prop3.6_2} $r(h)$ is conjugate to $r(g)$ under an inner automorphism restricted to the first four copies of $D_4$. On the fifth and sixth copy of $D_4$ the automorphisms $r(g)$ and $r(h)$ are outer and hence conjugate under an inner automorphism by Lemma~\ref{lem:innerouter}.

Hence $r(h)$ and $r(g)$ are conjugate under an inner automorphism of $D_4^6$.\qedhere
\end{enumerate}
\end{proof}
We conclude this section with the proof of its main result.
\begin{proof}[Proof of Proposition~\ref{prop:part3}]
In each case let $\hat\nu$ be a standard lift to $\Aut(V_L)$ of the lattice automorphism $\nu \in\Aut(L)$ specified in Proposition~\ref{prop:part1}. Since $\hat\nu$ has the required order, weight-one fixed-point Lie subalgebra and, if necessary, conformal weight of its twisted module, the existence of the automorphism is established.

We will show that any automorphism $g\in\Aut(V_L)$ with these properties is conjugate to $\hat\nu$. We showed in the previous lemma that $r(g)$ and $r(\hat\nu)$ are conjugate in $\Aut((V_L)_1)$ under an automorphism in $\pi^{-1}(H)$.

This automorphism in $\pi^{-1}(H)$ can be lifted to an automorphism in $\Aut(V_L)$, and using Lemma~\ref{lem:trivialonla} this implies that $g$ is conjugate in $\Aut(V_L)$ to $\e^{(2\pi\i)h_0}\hat\nu$ for some $h\in Q'$, the dual lattice of $Q$.
In fact, by Lemma~\ref{lem:conjproj} we may assume that $g=\e^{(2\pi\i)h_0}\hat\nu$ for some $h\in\pi_\nu(Q')$. Note that now $\e^{(2\pi\i)h_0}$ and $\hat\nu$ commute.

We aim to show that $\e^{(2\pi\i)h_0}=\id$. Since $L$ is unimodular, $\e^{(2\pi\i)h_0}=\id$ for $h\in L$ so that we only need to consider representatives $h$ for $\pi_\nu(Q')/\pi_\nu(L)=(Q^\nu)'/(L^\nu)'$. Moreover, we only need to consider those elements $h$ such that $g=\e^{(2\pi\i)h_0}\hat\nu$ has the same order as $\hat\nu$. This is the case if and only if $\ord(\hat\nu)h\in L$.

For the reader's convenience we list $L$, the cycle shape $s$ of $\nu$, the fixed-point sublattices $L^\nu$ and $Q^\nu$ (their genera, see \cite{CS99}, Chapter~15) and $(Q^\nu)'/(L^\nu)'$:
\begin{equation*}
\renewcommand{\arraystretch}{1.2}
\setlength{\tabcolsep}{.95\tabcolsep}
\begin{tabular}{rlllll}
No.\ & $L$               & $s$                  & $L^\nu$                   & $Q^\nu$                                 &$(Q^\nu)'/(L^\nu)'$\\\hline
 (1) & $N(A_9^2D_6)$     & $1^{-8}2^{16}$       & $\II_{8,0}(2_{\II}^{+6})$ & $\II_{8,0}(2_0^{+8})$                      &$\Z_2$          \\
 (2) & $N(A_7^2D_5^2)$   & $1^{-8}2^{16}$       & $\II_{8,0}(2_{\II}^{+8})$ & $\II_{8,0}(2_{\II}^{+6}4_1^{+1}16_7^{+1})$ &$\Z_4$          \\
 (3) & $N(A_8^3)$        & $1^{-8}2^{16}$       & $\II_{8,0}(2_{\II}^{+8})$ & $\II_{8,0}(2_{\II}^{+8}9^{+1})$            &$\Z_3$          \\
 (4) & $N(E_8^3)$        & $1^{-8}2^{16}$       & $\II_{8,0}(2_{\II}^{+8})$ & $\II_{8,0}(2_{\II}^{+8})$                  &$\{0\}$         \\
 (5) & $N(A_5^4D_4)$     & $1^{-8}2^{16}$       & $\II_{8,0}(2_{\II}^{+8})$ & $\II_{8,0}(2_{\II}^{-6}4_4^{-2}3^{-2})$    &$\Z_6$          \\
 (6) & $N(E_6^4)$        & $1^{-8}2^{16}$       & $\II_{8,0}(2_{\II}^{+4})$ & $\II_{8,0}(2_{\II}^{+4})$                  &$\{0\}$         \\
 (7) & $N(A_4^6)$        & $1^{-8}2^{16}$       & $\II_{8,0}(2_{\II}^{+8})$ & $\II_{8,0}(2_{\II}^{+8}5^{+2})$            &$\Z_5$          \\
 (8) & $N(A_2^{12})$     & $1^{-8}2^{16}$       & $\II_{8,0}(2_{\II}^{+8})$ & $\II_{8,0}(2_{\II}^{+8}3^{+4})$            &$\Z_3\times\Z_3$\\
 (9) & $N(A_{11}D_7E_6)$ & $1^{-8}2^{16}$       & $\II_{8,0}(2_{\II}^{+4})$ & $\II_{8,0}(2_{\II}^{+4})$                  &$\{0\}$         \\
(10) & $N(A_{11}D_7E_6)$ & $1^{-8}2^{16}$       & $\II_{8,0}(2_{\II}^{+6})$ & $\II_{8,0}(2_0^{+8})$                      &$\Z_2$          \\
(11) & $N(A_4^6)$        & $1^{-1}5^5$          & $\II_{4,0}(5^{+3})$       & $\II_{4,0}(5^{+3}25^{+1})$                 &$\Z_5$          \\
(12a)& $N(D_{10}E_7^2)$  & $1^22^{-9}4^{10}$    & $\II_{3,0}(2_3^{+3})$     & $\II_{3,0}(2_3^{+3})$                      &$\{0\}$         \\
(12b)& $N(D_{10}E_7^2)$  & $2^{-4}4^8$          & $\II_{4,0}(2_4^{+4})$     & $\II_{4,0}(2_4^{+4})$                      &$\{0\}$         \\
(13) & $N(D_{10}E_7^2)$  & $2^{-4}4^8$          & $\II_{4,0}(2_{\II}^{-2})$ & $\II_{4,0}(2_{\II}^{-2})$                  &$\{0\}$         \\
(14) & $N(E_6^4)$        & $1^32^{-3}3^{-9}6^9$ & $\{0\}$                   & $\{0\}$                                    &$\{0\}$         \\
(15) & $N(D_4^6)$        & $4^{-2}8^4$          & $\II_{2,0}(4_2^{+2})$     & $\II_{2,0}(8_2^{+2})$                      &$\Z_2$
\end{tabular}
\renewcommand{\arraystretch}{1}
\end{equation*}

Now we proceed case by case:
\begin{enumerate}
\item In this case $\pi_\nu(Q')/\pi_\nu(L)$ is isomorphic to $\Z_2$. For trivial $h$ we obtain $\rho(V_L(g))=1$ while for non-trivial $h$ we obtain $\rho(V_L(g))=5/4$. Hence, if we demand that $\rho(V_L(g))=1$, $g$ has to be conjugate to $\hat\nu$.
\item In this case $\pi_\nu(Q')/\pi_\nu(L)$ is isomorphic to $\Z_4$, but only the trivial element preserves the order of $\hat\nu$. Hence $g$ is conjugate to $\hat\nu$.
\item In this case $\pi_\nu(Q')/\pi_\nu(L)$ is isomorphic to $\Z_3$, but only the trivial element preserves the order of $\hat\nu$. Hence $g$ is conjugate to $\hat\nu$.
\item In this case $\pi_\nu(Q')/\pi_\nu(L)$ is trivial since $L=Q$, and so $g$ is conjugate to $\hat\nu$.
\item In this case $\pi_\nu(Q')/\pi_\nu(L)$ is isomorphic to $\Z_6$, but only the trivial element preserves the order of $\hat\nu$. Hence $g$ is conjugate to $\hat\nu$.
\item Also in this case $\pi_\nu(Q')/\pi_\nu(L)$ is trivial. Hence $g$ is conjugate to $\hat\nu$.
\item In this case $\pi_\nu(Q')/\pi_\nu(L)$ is isomorphic to $\Z_5$, but only the trivial element preserves the order of $\hat\nu$. Hence $g$ is conjugate to $\hat\nu$.
\item In this case $\pi_\nu(Q')/\pi_\nu(L)$ is isomorphic to $\Z_3\times\Z_3$, but only the trivial element preserves the order of $\hat\nu$. Hence $g$ is conjugate to $\hat\nu$.
\item In this case $\pi_\nu(Q')/\pi_\nu(L)$ is trivial. Hence $g$ is conjugate to $\hat\nu$.
\item In this case $\pi_\nu(Q')/\pi_\nu(L)$ is isomorphic to $\Z_2$. For trivial $h$ we obtain $\rho(V_L(g))=1$ while for non-trivial $h$ we obtain $\rho(V_L(g))=5/4$. Hence if we demand that $\rho(V_L(g))=1$, $g$ has to be conjugate to $\hat\nu$.
\item In this case $\pi_\nu(Q')/\pi_\nu(L)$ is isomorphic to $\Z_5$. For trivial $h$ we obtain $\rho(V_L(g))=1$ while for the four non-trivial choices of $h$ we obtain $\rho(V_L(g))=27/25$ or $\rho(V_L(g))=28/25$. Hence, if we demand that $\rho(V_L(g))=1$, $g$ has to be conjugate to $\hat\nu$.
\item[(12a)] In this case $\pi_\nu(Q')/\pi_\nu(L)$ is trivial. Hence $g$ is conjugate to $\hat\nu$.
\item[(12b)] In this case $\pi_\nu(Q')/\pi_\nu(L)$ is trivial. Hence $g$ is conjugate to $\hat\nu$.
\item[(13)] In this case $\pi_\nu(Q')/\pi_\nu(L)$ is trivial. Hence $g$ is conjugate to $\hat\nu$.
\item[(14)] In this case $\pi_\nu(Q')/\pi_\nu(L)$ is trivial (we see from the cycle shape of $\nu$ that $\pi_\nu$ is the zero map). Hence $g$ is conjugate to $\hat\nu$.
\item[(15)] In this case $\pi_\nu(Q')/\pi_\nu(L)$ is isomorphic to $\Z_2$. For trivial $h$ we obtain $\rho(V_L(g))=1$ while for the non-trivial choice of $h$ we obtain $\rho(V_L(g))=17/16$. Hence, if we demand that $\rho(V_L(g))=1$, $g$ has to be conjugate to $\hat\nu$.
\end{enumerate}

In order to prove the second statement of the proposition we note that the inverse orbifold automorphism $\amgis\in\Aut(V^{\orb(\sigma_h)})$ with respect to the inner automorphism $\sigma_h$ (see table on p.~\pageref{page:table}) has the same order and weight-one fixed-points as $\hat\nu$. In all cases except (1), (10), (11) and (15) this shows that $\hat\nu\in\Aut(V_L)$ and $\amgis\in\Aut(V^{\orb(\sigma_h)})$ are conjugate under the isomorphism $V^{\orb(\sigma_h)}\cong V_L$.

For the four remaining cases we still need to show that $\rho(V^{\orb(\sigma_h)}(\amgis))=1$. The $\amgis$-twisted $V^{\orb(\sigma_h)}$-module $V^{\orb(\sigma_h)}(\amgis)$ decomposes as
\begin{equation*}
V^{\orb(\sigma_h)}(\amgis)=\bigoplus_{i\in\Z_n}W^{(i,1)},
\end{equation*}
as described in Section~\ref{sec:invorb}. We know that $\rho(V(\sigma_h^i))\geq 1$ for all $i\in\Z_n\setminus\{0\}$, and hence, since $V(\sigma_h^i)=\bigoplus_{j\in\Z_n}W^{(i,j)}$, $\rho(W^{(i,j)})\geq 1$ for $i\in\Z_n\setminus\{0\}$ and $j\in\Z_n$. Moreover, since $V=\bigoplus_{j\in\Z_n}W^{(0,j)}$ is of CFT-type, $\rho(W^{(0,j)})\geq 1$ for $j\in\Z_n\setminus\{0\}$. Hence, it suffices to show that $\rho(W^{(0,1)})=1$ or in other words that $\{v\in V_1\,|\,\sigma_hv=\xi_nv\}$ is non-empty. This is easily verified for the cases at hand.
\end{proof}

\section{Main Result}\label{sec:main}
Combining the results of Propositions \ref{prop:part1}, \ref{prop:part2} and \ref{prop:part3} and following the procedure described in Section~\ref{sec:approach} we arrive at the main result of this text:
\begin{thm}[Uniqueness]\label{thm:main}
\maintheorem
\end{thm}
The uniqueness of the case $B_8E_{8,2}$ (case~62) has already been proved in \cite{LS15a} using framed \voa{}s. Moreover, after the completion of this work, Lam and Shimakura also showed the uniqueness of $A_{4,5}^2$ (case~9) and $A_{1,2}A_{3,4}^3$ (case~7) starting from the Leech lattice \voa{} \cite{LS17}.

This leaves 9 cases on Schellekens' list, for which the uniqueness has yet to be proved, including the case $V_1=\{0\}$ (see Remark~6.8 in \cite{LS17} for an explicit list).

\bibliographystyle{alpha_noseriescomma}
\bibliography{quellen}{}

\end{document}